\documentclass[a4paper, 11pt]{article}
\usepackage[sort&compress]{natbib}
\usepackage[margin=1.0in]{geometry}

\usepackage[utf8x]{inputenc}
\usepackage[english]{babel}
\usepackage{graphicx, amsfonts,amsmath,amssymb,amsthm}
\usepackage{commands,tikzSettings}
\usepackage{longtable}

\setcitestyle{authoryear,open={(},close={)}} 

\author{Victor Cohen$^1$, Axel Parmentier$^2$}
\title{Future memories are not needed for large classes of POMDPs}
\date{%
	$^1$Swiss Data Science Center (SDSC), ETH Zürich, 8006 Zürich, Switzerland\\%
    $^2$CERMICS, \'Ecole des Ponts, Marne-la-Vall\'ee, France\\%
    \today}
\begin{document}
\maketitle

\abstract{
    Optimal policies for partially observed Markov decision processes (POMDPs) are history-dependent: Decisions are made based on the entire history of observation.
    Memoryless policies, which take decisions based on the last observation only, are generally considered useless in the literature because we can construct POMDP instances for which optimal memoryless policies are arbitrarily worse than history-dependent ones.
    Our purpose is to challenge this belief.
    We show that optimal memoryless policies can be computed efficiently using mixed integer linear programming (MILP), 
    and perform reasonably well on a wide range of instances from the literature.
    When strengthened with valid inequalities, the linear relaxation of this MILP provides high quality upper-bounds on the value of an optimal history dependent policy.
    Furthermore, when used with a finite horizon POMDP problem with memoryless policies as rolling optimization problem, a model predictive control approach
    leads to an efficient history-dependent policy, which we call the short memory in the future (SMF) policy. Basically, the SMF policy leverages these memoryless policies to build an approximation of the Bellman value function.
    Numerical experiments show the efficiency of our approach on benchmark instances from the literature.
}



\section{Introduction}
\label{sec:intro}

Decision makers often have to control a system they do not fully observe.
\emph{Partially observed Markov decision processes }(POMDP) enable to model such problems.
They can be described as follows.
At each time $t$ in $\{1,\ldots,n\}$ with $n$ in $\bbZ_+$ or $n=+\infty$, the system is in a state $s \in \calS$.
The decision maker does not observe $s$, but has access to a noisy observation $o \in \calO$ emitted with a known probability $p(o|s)$.
Based on the information available, the decision maker takes a decision $a$ in $\calA$. The decision maker perceives a reward $r(s,a)$ and the system then evolves toward the next state $s'$ with probability $p(s'|s,a)$.
In this paper, we suppose that $\calS$, $\calO$, and $\calA$ are finite.
A \emph{policy} $\bfdelta$ models how decisions are taken.
The decision maker has access to all the information available: At time $t$, a policy is given by the conditional probability $\delta_t(a_t|h_t)$ of taking decision $a_t$ at time $t$ given the history of observations and actions $h_t=(o_0,a_0,\ldots,o_t )$. Such a policy is said \emph{history-dependent} because of the decision is taken based on variable $h_t$ at every time $t$. 
The purpose of the POMDP problem is to find a history-dependent policy that maximizes the expected reward.
When the horizon $n$ is infinite, we usually consider the discounted expected reward, which gives the problem 
\begin{align}\label{pb:POMDP}
	\max_{\bfdelta} \bbE_{\bfdelta} \Bigl[ \sum_{t=0}^{+\infty} \gamma^{t} r(S_t,A_t)\Bigr]
\end{align}
where $\bbE_{\bfdelta}$ denotes the distribution on the state, observation, and action spaces implied by policy $\bfdelta$.
The optimization problem can also be defined with a finite horizon $n$ and a discount factor set to one, i.e., $\gamma=1$.
It has been shown that all the information until time $t$ is contained in the \emph{belief state}, which is the conditional distribution $\bbP(s_t|h_t)$ on the hidden state given the history variable $h_t$ \citep[Theorem 4]{Eckles1968}. The belief state is a sufficient statistic of the history of information.
An optimal policy can therefore be found as a solution of a (large dimension) Bellman equation in the belief state \citep{Sondik1978,Sondik1973}. 
Most algorithms build value functions that are approximate solutions of this Bellman equation \citep{Hauskrecht2000,Shani2013}.

Finding an optimal history-dependent policy is computationally difficult, not the least because tabulating a history-dependent policy requires exponential memory.
Indeed, the problem with finite horizon is P-SPACE hard \citep[Theorem 6]{Tsitsiklis1987} and is undecidable with discounted infinite horizon \citep[Theorem 4.3]{Madani1999}.
One way of simplifying the problem is to restrict oneself to memoryless policies.
A \emph{memoryless} policy uses exclusively observation $o_t$ to take decision $a_t$ and can be modeled using a conditional probability $\delta_t(a_t|o_t)$.
The problem of finding an optimal memoryless policy is indeed only NP-hard \citep{Littman94memoryless}.
However, an optimal memoryless policy can be arbitrarily worse than an optimal history-dependent policy (see the ``Maze'' example in \cite{Littman94memoryless}).
For this reason, memoryless policies seem to be considered worthless and have received relatively little attention in the literature \citep{Li2011,Singh98,Steckelmacher2017}.
Our purpose is to challenge this belief and to show that memoryless policies are useful as ``intermediate steps'' in the search of history-dependent policies.
\begin{enumerate}
	\item We introduce a history-dependent policy that approximates the value function by computing an optimal memoryless policy for a finite rolling horizon starting from the current belief state.
	\item We introduce a MILP formulation for the problem of finding an optimal memoryless policy for a finite horizon, and strengthen it with valid inequalities. Its linear relaxation provides a collection of upper bounds on the value of a history-dependent policy.
\end{enumerate}
To the best of our knowledge, the \texttt{POMDP.org} website \citep{Cassandra2003ASO} provides the largest open library of POMDP instances.
We extensively benchmark our algorithms on finite and infinite horizon instances from this library and reach the following conclusions. 
\begin{enumerate}[resume]
	\item Our MILP formulation provides a tractable way of computing optimal memoryless policies. We were surprised to observe that, for $60 \%$ of the instances , an optimal memoryless policy is in a $20 \%$ of gap with an optimal history dependent policy.
	\item When we use our MILP formulation to compute the optimal memoryless policies, our history-dependent policy is competitive with state-of-the-art algorithms. 
\end{enumerate}
The rest of the paper is organized as follows. 
In Section~\ref{sec:history_dependent_policy}, we formally introduce the POMDP problem, the definition of a belief state policy and the main principles of our history-dependent policy, which is based on the computation of a memoryless policy. Then, Section~\ref{sec:milp} details the mathematical program that gives the intermediate memoryless policy, as well as the theorems guaranteeing that its linear relaxation is an upper bound on the value of an optimal history-dependent policy. Finally, Section~\ref{sec:num} provides numerical experiments showing the efficiency of our approach on benchmark instances.
All the proofs of the results of this paper are available in Appendix~\ref{app:validInequalities:proofs}.

\section{History dependent policy}
\label{sec:history_dependent_policy}


We introduce the \emph{belief state space} $\calB$ as the set of probability distributions (the simplex) on $\calS$, i.e., 
$$\calB = \big\{\bfp: \calS \rightarrow [0,1],\ \text{such that } p(s) \geq 0 \text{ for all } s\in \calS, \ \text{and} \ \sum_{s\in\calS} p(s)=1  \big\}.$$
Since the history of information at time $t$ can be summarized by the belief state $\bfb_t$, a history-dependent policy is given by the probability $\delta_t(a_t|\bfb_t)$ of taking action $a_t$ given the belief $\bfb_t$ at time $t$.

\paragraph{Bellman optimality equation.}
As mentioned in the introduction, a POMDP can be seen as a MDP in the belief state space. It follows that we can write the Bellman's principle of optimality in the belief state space.
Given an optimal policy $\bfdelta^*=\left(\delta_0,\ldots \right)$ of Problem~\eqref{pb:POMDP}.
For any belief $\bfb$ in $\calB$, we denote by $v_t(\bfb)$ the expected value of the reward $\bbE_{\bfdelta^*} \left[ \sum_{t'=t}^{+\infty} \gamma^{t'-t}r(S_{t'},A_{t'}) | S_0 \sim \bfb\right]$ starting in belief $\bfb$ at time $t$.
Since policy $\bfdelta^*$ is optimal, the value function satisfies the following Bellman's equation \citep{Bellman1957} at any time $t$
\begin{equation}\label{eq:BellmanBelief}
    v_t(\bfb) = \max_{a \in A} \underbrace{\left(\sum_{s \in \calS} b(s)r(s,a) + \gamma \sum_{o \in \calO} p(o|a,\bfb) v_{t+1}(f(\bfb,a,o)) \right)}_{\calQ_t(a,\bfb)},
\end{equation}
where $f \colon \calB \times \calA \times \calO \rightarrow \calB$ is the belief state transition function, i.e., $f(\bfb,a,o)$ indicates the belief state at time $t+1$, which is defined below. 
Note that the conditional probability $p(o|a,\bfb)$ can be expressed using the Bayes formula $p(o|a,\bfb) = \sum_{s,s' \in \calS} p(o|s)p(s|s',a)b(s')$.
Given the optimal value functions $(v_t)_{t \in \{0,\ldots,n\}}$, an optimal policy $\bfdelta^* = (\delta_t)_{t \in \{0,\ldots,n\}}$ is obtained by assigning $\delta_t(a|\bfb)=1$ if action $a$ belongs to $\argmax_{a \in \calA} \calQ_t(a,\bfb)$.
In the rest of the paper, given any initial belief $\bfb$ in $\calB$ we denote by $v^*(\bfb)$ the optimal value of the POMDP starting in belief $\bfb$, i.e., $v^*(\bfb) := v_0(\bfb)$, and with the same state transition probability distribution.
And we introduce the Q-function $\calQ_t(\bfb,a)$ as the expression in the maximum operand on the right-hand side of Bellman equation~\eqref{eq:BellmanBelief}.

In the discounted infinite horizon setting, Bellman equation~\eqref{eq:BellmanBelief} can be simplified by considering \emph{stationary policies}. 
A policy $\bfdelta$ is stationary if $\bfdelta = (\delta, \delta,\ldots)$.
\citet[Theorem 6]{RossSheldonM1997Apmw} guarantees that optimal value function $v_0$ satisfies
\begin{align*}
    v^*(\bfb) = \max_{a' \in \calA} \underbrace{\left(\sum_{s \in \calS} b(s)r(s,a') + \gamma \sum_{o \in \calO} p(o|a',\bfb) v^*(f(\bfb,a',o)) \right)}_{\calQ(\bfb,a)},
\end{align*}
and we can deduce from this equation that there exists an optimal policy that is stationary.

\paragraph{Belief state update.}
The belief transition function $f$ mentioned above has an explicit formula \citep[Eq. (2.2)]{POMDPMonahan1982}. Given a belief $\bfb$ in $\calB$, an action $a$ in $\calA$ and an observation $o$ in $\calO$, we have 
\begin{align}\label{eq:belief_state_update}
    f(\bfb,a,o)(\tilde{s})= \bbP(S_{t+1} = \tilde{s}| A_t=a,O_{t+1}=o,S_t\sim\bfb)= \frac{p(o_{t+1}|\tilde{s})\sum_{s\in \calS} p(\tilde{s}|s, a_t)b(s)}{\sum_{s,s' \in \calS} p(o_{t+1}|s')p(s'|s, a_t)b(s)}
\end{align}

\paragraph{MDP approximation.}
Before introducing our algorithm, we recall the definition of the \emph{MDP approximation} \citep[Section 4.1]{Hauskrecht2000}, which corresponds to the case where the decision maker has access to the system state. In this case, the POMDP becomes a MDP and the policy is of the form $\delta(a|s)$ for any action $a$ in $\calA$ and any state $s$ in $\calS$. 
Given an initial state $s$ in $\calS$, we denote by $v_{\rmMDP}(s)$ the optimal value of a MDP starting in state $s$ over an infinite discounted horizon. It satisfies the Bellman equation \citep{Bellman1957}
\begin{align}\label{pb:BellmanBeliefMDP}
    v_{\rmMDP}(s) = \max_{a'\in \calA} \left(r(s,a') + \gamma \sum_{s' \in \calS} p(s'|s,a') v_{\rmMDP}(s') \right).
\end{align}
There are several methods to compute $v_{\rmMDP}$: linear programming, value iteration, policy iteration (see e.g.~the book of~\citet{Puterman1994} for more details on these methods).




\paragraph{Our approach.}
Computing the value function $v^*(\bfb)$ is intractable as soon as the state and observation spaces are of moderate size.
We therefore suggest using a heuristic policy. 
As mentioned in the introduction, the performances of memoryless policies as well as their lower level of complexity lead us to consider them as a tool to approximate the value function.
We therefore suggest approximating~$v^*(\bfb)$ by~$\hat v_{\rmSMF}^T(\bfb)$, which we define as the value of optimal memoryless policy starting in belief $\bfb$, with a finite horizon $T$ and a tail reward corresponding to the MDP value function, 



\begin{equation}\label{eq:memorylessFromBeliefValueFunction}
    \hat v_{\rmSMF}^T(\bfb) = \max_{a' \in \calA}\underbrace{\max
_{\bfdelta\in \Deltaml^{T}}
    \bbE_{\bfdelta}\Big[ \sum_{t=0}^{T} \gamma^{t}r(S_{t},A_{t}) + \gamma^{T+1}v_{\rmMDP}(S_{T+1}) \Big| S_0 \sim\bfb, A_0 = a\Big]
    }_{\hat Q_{\rmSMF}^T(\bfb,a'): \ \substack{\text{Optimal memoryless value of POMDP} \\ \text{with $v_{\rmMDP}$ as tail reward}}},
\end{equation}
where $\ind_a(a')$ is the indicator function equal to $1$ if $a'=a$ and $0$ otherwise, and $\Deltaml^{T'}$ is the set of memoryless policies over $T'$ time steps, which will be formally defined in the next section. 
The acronym SMF stands for short memory in the future.
Given a finite horizon $T$, we introduce the Q-function $\hat \calQ_{\rmSMF}^T$ as the expression in the maximum operand on the right-hand side of~\eqref{eq:memorylessFromBeliefValueFunction}. 
The finite horizon $T$ is also called \emph{rolling horizon} in the optimal control theory \citep{Bertsekas2005}.
One may observe that there is no dependence in $t$, i.e., the time at which the decision is taken. Indeed, since we assume that the parameters are stationary in the present model, the optimal value only depends on the initial belief state $\bfb$.

\paragraph{Our policy.}
Given a finite horizon $T$ in $\bbZ_{+}$, we define the Short Memory in the Future (SMF) policy $\bfdelta_{\mathrm{SMF(T)}}$ as follows:
\begin{align}\label{eq:history_dep_pol:policySMF}
    \delta_{\mathrm{SMF}(T)}(a|\bfb) = \begin{cases}
                    1 \ \text{if} \ a  \in  \displaystyle\argmax_{a' \in \calA} \hat Q_{\rmSMF}^T(\bfb,a') \\
                    0 \ \text{otherwise}
                    \end{cases}
\end{align}
In Definition~\eqref{eq:history_dep_pol:policySMF}, there is an abuse of notation because it defines a collection of SMF policies since there may be more than one element in $\argmax_{a' \in \calA} \hat Q_{\rmSMF}^T(\bfb,a')$. We emphasize that in this paper, we focus on deterministic policies, i.e., $\delta_{\rmSMF(T)}(a|\bfb)$ belongs to $\{0,1\}$.
It takes its decision based on all the information available on the past, but ignores the fact that decision in the future are taken based on all the available information, which is the reason why we say that our policy has a short memory for the future decisions.
Any finite horizon memoryless solver can be used to compute SMF policy. We suggest solving this memoryless problem with the MILP formulation introduced in the next section.

Given an initial belief $\bfb$ in $\calB$ and a finite horizon $T$ in $\bbZ_{+}$, we denote by $v_{\rmSMF(T)}(\bfb)$ the value of SMF policy. 
We underline that $v_{\rmSMF(T)}(\bfb)$ is not equal to $\hat v_{\rmSMF}^T(\bfb)$, the approximation of the value function used by the SMF policy. 
Since $\delta_{\rmSMF(T)}$ is a feasible policy, we have $v_{\rmSMF(T)}(\bfb) \leq v^*(\bfb)$.
The numerical experiments in Section~\ref{sec:num} show that SMF policy performs well on a large number of benchmark instances compared to a state-of-the-art POMDP solver.

\section{Mathematical programming for memoryless POMDPs}
\label{sec:milp}


The aim of this section is to explain how to compute an optimal memoryless policy for a POMDP with finite horizon, which is required for the history-dependent policy of Section~\ref{sec:history_dependent_policy}.
To do so, we focus on the optimization problem, and introduce the notation $w_{\rmml}^T(\bfb)$ for its value.
\begin{equation}\label{pb:POMDPmlfiniteHorizon}
    w_{\rmml}^T(\bfb) = \max_{a \in \calA, \bfdelta \in \Deltaml^T} \bbE_{\bfdelta} \Bigl[ \sum_{t=0}^{T} r_t(S_t,A_t) | S_0 \sim \bfb, A_0 = a\Bigr]
\end{equation}
where $\Deltaml^{T}$ denotes the set of memoryless policies over a finite horizon $T$,
\begin{align}\label{eq:problem:def_policy_set}
\displaystyle \Deltaml^T = \bigg\{ \bfdelta \in \bbR^{(T+1) \times \calA \times \calO } \colon \sum_{a \in \calA} \delta^t_{a|o} = 1 \ \mathrm{and} \ \delta^t_{a|o} \geq 0,\enskip \text{for all } o \in \calO, \enskip a \in \calA, t \in [T] \bigg\},
\end{align}
$\bfb$ in $\calB$ is an initial belief state, $a_0$ in $\calA$ is an initial action, and $r_t$ is a reward function for every $t$ in $[T]$.
In the notation $w_{\rmml}^T(\bfb)$, the subscript "ml" refers to memoryless. Similarly, we denote by $\Delta^T$ the set of history-dependent policies over a finite horizon $T$, and $w^T(\bfb)$ the optimal value of the corresponding POMDP problem with history dependent policies.
In Definition~\eqref{eq:problem:def_policy_set}, the element $\delta_{a|o}^t$ can be interpreted as the conditional probability of taking action $a$ given $o$ at time $t$.
The reward function $r_t$ is arbitrary, and not necessarily the one introduced in Section~\ref{sec:history_dependent_policy}. 
In general, the POMDP problem with finite horizon is defined with a reward function which does not depend on time. In our case, the results of this section applies as well to the most general case with time dependent reward function.

In Section~\ref{sub:milp:NLP} and Section~\ref{sub:milp:MILP}, we introduce respectively a non-linear program (NLP) and a mixed integer linear program (MILP) to solve Problem~\eqref{pb:POMDPmlfiniteHorizon}.
Then, in Section~\ref{sub:milp:valid_cuts} we introduce the valid inequalities that enable to tighten the linear relaxation of our MILP.

\subsection{An exact Nonlinear Program}
\label{sub:milp:NLP}

We introduce the following nonlinear program (NLP) with a collection of variables\\
$\bfmu = \left((\mu_s^t)_s,(\mu_{soa}^t)_{s,o,a},(\mu_{sa}^t)_{s,a}\right)_{t}$, $\bfdelta = \left( \left(\delta_{a|o}^t\right)_{a,o} \right)_{t}$.

\begin{subequations}\label{pb:milp:NLP_pomdp}
\begin{alignat}{2}
\max_{\bfmu, \bfdelta}  \enskip & \sum_{t=0}^T \sum_{\substack{s \in \calS,a \in \calA}} r_t(s,a) \mu_{sa}^t  & \quad & \label{eq:milp:NLP_obj_function}\\
\mathrm{s.t.} \enskip
 & \delta_{a|o}^0 = \delta_{a|o'}^0 & \forall o,o' \in \calO, a \in \calA \label{eq:milp:NLP_initial_policy} \\
 & \mu_{s}^{0} = \bfb(s) &  \quad  \forall s \in \calS \label{eq:milp:NLP_state_initial} \\
 & \mu_{s}^{t} = \sum_{a \in \calA} \mu_{sa}^t &  \quad  \forall s \in \calS, t \in [T] \label{eq:milp:NLP_consistency_s_t}\\
 & \mu_{sa}^t =  \sum_{o \in \calO} \mu_{soa}^t &  \quad  \forall s \in \calS, a \in \calA, t \in [T] \label{eq:milp:NLP_consistency_sa}\\
 & \mu_s^{t+1} = \sum_{s' \in \calS, a' \in \calA} p(s|s',a')\mu_{s'a'}^t  &  \quad  \forall s \in \calS, t \in [T] \label{eq:milp:NLP_consistency_s} \\
  & \mu_{soa}^t = \delta^t_{a|o} p(o|s) \mu_s^t & \quad \forall s \in \calS, o \in \calO, a \in \calA, t \in [T] \label{eq:milp:NLP_indep_action} \\
 & \bfdelta \in \Deltaml^T, \bfmu \geq 0 \label{eq:milp:NLP_constraint_policy} 
\end{alignat}
\end{subequations}

Given a policy $\bfdelta \in \Deltaml^T$, we say that $\bfmu$ is the vector of \emph{moments} of the probability distribution $\bbP_{\bfdelta}$ induced by $\bfdelta$ when
\begin{subequations}\label{eq:milp:def_moments}
	\begin{alignat}{2}
		&\mu_s^t = \bbP_{\bfdelta}(S_t=s), & & \quad \forall s \in \calS, \forall t \in [T+1]  \label{eq:milp:def_moments_s} \\
		&\mu_{soa}^t = \bbP_{\bfdelta}(S_t=s, O_t=o, A_t=a), & & \quad \forall s \in \calS, o \in \calO, a \in \calA, \forall t \in [T] \label{eq:milp:def_moments_soa}\\
		&\mu_{sa}^t = \bbP_{\bfdelta}(S_t=s, A_t=a), & & \quad \forall s \in \calS, a \in \calA, \forall t \in [T] \label{eq:milp:def_moments_sa}
	\end{alignat}
\end{subequations}
In the definition of vector of moments~\eqref{eq:milp:def_moments}, we omit the conditioning of the probability distributions on the event $\{\bfb_0 = \bfb \}$.
Thanks to the properties of probability distributions, such vector of moments~\eqref{eq:milp:def_moments} of $\bbP_{\bfdelta}$ satisfies the constraints of Problem~\eqref{pb:milp:NLP_pomdp}. Conversely, given a feasible solution of Problem~\eqref{pb:milp:NLP_pomdp}, Theorem~\ref{theo:milp:NLP_optimal_solution} ensures that $\bfmu$ is the vector of moments of $\bbP_{\bfdelta}$. 
We denote by $z_{\rmml}^{T}(\bfb)$ the optimal value of Problem~\eqref{pb:milp:NLP_pomdp}. 

\begin{restatable}{theo}{NLPoptimalSolution}\label{theo:milp:NLP_optimal_solution}
    Let $(\bfmu, \bfdelta)$ be a feasible solution of NLP~\eqref{pb:milp:NLP_pomdp}. Then $\bfmu$ is the vector of moments of the probability distribution $\bbP_{\bfdelta}$ induced by $\bfdelta$, and $(\bfmu, \bfdelta)$ is an optimal solution of NLP~\eqref{pb:milp:NLP_pomdp} if and only if $\bfdelta$ is an optimal policy of Problem~\eqref{pb:POMDPmlfiniteHorizon}. In particular, $w_{\rmml}^T(\bfb) = z_{\rmml}^{T}(\bfb)$. 
\end{restatable}

One can observe that Constraints~\eqref{eq:milp:NLP_initial_policy} enforce the initial policy $\delta^0$ to be independent from the initial observation. 
Despite the recent advances of the nonlinear solvers, numerical experiments in Section~\ref{sec:num} show that NLP solver perform poorly on benchmark instances. Indeed, either these solvers are based on heuristics \citep[\texttt{KNITRO}]{Byrd2006} which find solutions far from the optimum, or the solvers are based on exact algorithms \cite{gurobi}, which have high computational cost and take much more time.

\subsection{Turning the nonlinear program into an MILP}
\label{sub:milp:MILP}

In this section, we explain how to derive a MILP that computes $w_{\rmml}^T(\bfb)$.
We define the set of \emph{deterministic memoryless policies} $\Deltaml^{T,\rm{d}} = \Deltaml^{T} \cap \{0,1\}^{(T+1) \times \calA \times \calO}$.
The following proposition states that we can restrict our policy search in~\eqref{pb:POMDPmlfiniteHorizon} to the set of deterministic memoryless policies.
\begin{prop}{\cite[Proposition 1]{Bagnell2004}}\label{prop:milp:det_policies}
	For any belief $\bfb$ in $\calB$, there always exists an optimal policy for Problem~\eqref{pb:POMDPmlfiniteHorizon} that is deterministic, i.e.,
	\begin{equation*}\label{eq:milp:prop:det_policies}
    	\max_{a \in \calA,\bfdelta \in \Deltaml^{T}} \bbE_{\bfdelta}\Bigl[ \sum_{t=0}^{T} r_t(S_t,A_t) | \bfb, a \Bigr] = \max_{a \in \calA,\bfdelta \in \Deltaml^{T,\rm{d}}} \bbE_{\bfdelta} \Bigl[ \sum_{t=0}^{T} r_t(S_t,A_t) | \bfb, a\Bigr].
	\end{equation*}
\end{prop}
Theorem~\ref{theo:milp:NLP_optimal_solution} ensures that Problem~\eqref{pb:POMDPmlfiniteHorizon} and Problem~\eqref{pb:milp:NLP_pomdp} are equivalent, and in particular admit the same optimal solution in terms of $\bfdelta$. However, Problem~\eqref{pb:milp:NLP_pomdp} is hard to solve due to the nonlinear constraints~\eqref{eq:milp:NLP_indep_action}.
By Proposition~\ref{eq:milp:prop:det_policies}, we can add the integrality constraints of $\Deltaml^{\rm{d}}$ in~\eqref{pb:milp:NLP_pomdp}, and, by a classical result in integer programming, we can turn Problem~\eqref{pb:milp:NLP_pomdp} into an equivalent MILP: It suffices to replace constraint~\eqref{eq:milp:NLP_indep_action} by the following McCormick inequalities~\citep{Mccormick1976}.
\begin{subequations}\label{eq:milp:McCormick_linearization}
	\begin{alignat}{2}
		&\mu_{soa}^t \leq p(o|s) \mu_{s}^t & \quad \forall s \in \calS, o \in \calO, a \in \calA, t \in [T] \label{eq:milp:MILP_McCormick_1} \\
 		&\mu_{soa}^t \leq \delta^t_{a|o} & \quad  \forall s \in \calS, o \in \calO, a \in \calA, t \in [T] \label{eq:milp:MILP_McCormick_2} \\
 		&\mu_{soa}^t \geq p(o|s) \mu_{s}^t  + \delta^t_{a|o} - 1 & \quad \forall s \in \calS, o \in \calO, a \in \calA, t \in [T]. \label{eq:milp:MILP_McCormick_3}
	\end{alignat}
\end{subequations}
For convenience, we denote by $\mathrm{McCormick}\left(\bfmu, \bfdelta \right)$ the set of McCormick linear inequalities~\eqref{eq:milp:McCormick_linearization}. 
Thus, by using McCormick's linearization on constraints~\eqref{eq:milp:NLP_indep_action}, we get that~\eqref{pb:POMDPmlfiniteHorizon} is equivalent to the following MILP: 
\begin{equation}
\begin{aligned}\label{pb:milp:MILP_pomdp}
\max_{\bfmu, \bfdelta}  \enskip & \sum_{t=0}^T \sum_{\substack{s \in \calS \\ a \in \calA}} r_t(s,a) \mu_{sa}^t  & \quad & \\
\mathrm{s.t.} \enskip
 & \bfmu \ \mathrm{satisfies}~\eqref{eq:milp:NLP_initial_policy}-\eqref{eq:milp:NLP_consistency_s}\\
 & \mathrm{McCormick}\big(\bfmu,\bfdelta\big) \\
 & \bfdelta \in \Deltaml^{T,\rm{d}},  \bfmu \geq 0. 
\end{aligned}
\end{equation}

\subsection{Improving upper bound with valid inequalities}
\label{sub:milp:valid_cuts}


In this section, we introduce valid inequalities for MILP~\eqref{pb:milp:MILP_pomdp} that give a tighter linear relaxation. In addition, these inequalities will help to compute a good upper bound of the optimal history-dependent value.

We start by explaining why the linear relaxation of MILP~\eqref{pb:milp:MILP_pomdp} is not sufficient to define a feasible solution of Problem~\eqref{pb:milp:NLP_pomdp}.
It turns out that given a feasible solution $(\bfmu,\bfdelta)$ of the linear relaxation of MILP~\eqref{pb:milp:MILP_pomdp}, the vector $\bfmu$ is not necessarily the vector of moments of the probability distribution $\bbP_{\bfdelta}$ induced by $\bfdelta$. Indeed, when the coordinates of the vector $\bfdelta$ are continuous variables, the McCormick's constraints~\eqref{eq:milp:McCormick_linearization} are, in general, no longer equivalent to bilinear constraints~\eqref{eq:milp:NLP_indep_action}. Then, $(\bfmu,\bfdelta)$ is not necessarily a feasible solution of Problem~\eqref{pb:milp:NLP_pomdp} anymore, which implies that $\bfmu$ is not necessarily the vector of moments of the probability distribution $\bbP_{\bfdelta}$.
Actually, we can reduce the feasible set of the linear relaxation of MILP~\eqref{pb:milp:MILP_pomdp} by adding valid inequalities.
To do so, we introduce new variables $\left((\mu_{s'a'soa}^t)_{s',a',s,o,a}\right)_{t}$ and the inequalities

\begin{subequations}\label{eq:milp:Valid_cuts_pomdp}
    \begin{alignat}{2}
        &\sum_{s'\in \calS, a' \in \calA} \mu_{s'a'soa}^t = \mu_{soa}^{t}, \quad &\forall s \in \calS, o \in \calO, a \in \calA, \label{eq:milp:Valid_cuts_pomdp_consistency1}\\
        &\sum_{a \in \calA} \mu_{s'a'soa}^t = p(o|s)p(s|s',a')\mu_{s'a'}^{t-1}, \quad &\forall s',s \in \calS, o \in \calO, a' \in \calA, \label{eq:milp:Valid_cuts_pomdp_consistency2}\\
        &\mu_{s'a'soa}^t = p(s|s',a',o)\sum_{\ovs \in \calS} \mu_{s'a'\ovs oa}^t, \quad  &\forall s',s \in \calS, o \in \calO, a',a \in \calA, \label{eq:milp:Valid_cuts_pomdp_main}
    \end{alignat}
\end{subequations} 
\noindent where we use the constants $$p(s|s',a',o) = \bbP(S_t=s | S_{t-1}=s',A_{t-1}=a',O_t=o),$$
for any $s,s' \in \calS$, $a'\in \calA$ and $o \in \calO$.
Note that $p(s|s',a',o')$ does not depend on the policy $\bfdelta$ and can be easily computed during a preprocessing using Bayes rules. Therefore, constraints in~\eqref{eq:milp:Valid_cuts_pomdp} are linear.

\begin{restatable}{prop}{ValidInequalities}\label{prop:milp:valid_cuts_pomdp}
    Inequalities \eqref{eq:milp:Valid_cuts_pomdp} are valid for \textup{MILP}~\eqref{pb:milp:MILP_pomdp}, and there exists a solution $\bfmu$ of the linear relaxation of~\eqref{pb:milp:MILP_pomdp} that does not satisfy constraints~\eqref{eq:milp:Valid_cuts_pomdp}. 
\end{restatable}

The MILP formulation obtained by adding inequalities~\eqref{eq:milp:Valid_cuts_pomdp} in MILP~\eqref{pb:milp:MILP_pomdp} is an extended formulation, and has many more constraints than the initial MILP~\eqref{pb:milp:MILP_pomdp}. 
Its linear relaxation therefore takes longer to solve and the use of these linear inequalities may slow down the resolution on large scale instances.
In Appendix~\ref{app:validInequalities:probaInterpretation}, we show that valid inequalities~\eqref{eq:milp:Valid_cuts_pomdp} have also a probabilistic interpretation. Indeed, these equations correspond to conditional independences between the random variables of the problem, that are not induced in the linear relaxation of MILP~\eqref{pb:milp:MILP_pomdp}.

\paragraph{Improving the bounds on finite horizon.}
Our main use of the valid inequalities is to obtain an upper bound on the optimal value of history-dependent policies. 
Indeed, it turns out that solving the linear relaxation of MILP~\eqref{pb:milp:MILP_pomdp} with valid inequalities~\eqref{eq:milp:Valid_cuts_pomdp} enables to obtain a tighter upper bound.
In addition, Theorem~\ref{theo:milp:MDP_approx_equivalence_finiteHorizon} below states that the linear relaxation of MILP~\eqref{pb:milp:MILP_pomdp} is equivalent to the MDP approximation of POMDP with finite horizon.
We denote respectively by $z_{\rmRc}^{T}(\bfb)$ and $z_{\rmR}^{T}(\bfb)$ the value of the linear relaxation of MILP~\eqref{pb:milp:MILP_pomdp} with and without valid inequalities~\eqref{eq:milp:Valid_cuts_pomdp}.
 
\begin{restatable}{theo}{MDPequivalence}\label{theo:milp:MDP_approx_equivalence_finiteHorizon}
    For any belief $\bfb$ in $\calB$ and any horizon $T$ in $\bbZ_{+}$, the linear relaxation of MILP~\eqref{pb:milp:MILP_pomdp} is equivalent to the MDP approximation with finite horizon $T$, and, the following inequalities hold:
    \begin{equation}\label{eq:milp:inequality_information}
        w_{\rmml}^{T}(\bfb) \leq w^{T}(\bfb) \leq z_{\rmRc}^{T}(\bfb) \leq z_{\rmR}^{T}(\bfb).
    \end{equation}
\end{restatable}

Inequalities~\eqref{eq:milp:inequality_information} tell us that we are able to give a bound on the optimality gap between memoryless policies and history-dependent policies. 
This gap is an indicator of the value of information loss by restricting to memoryless policies.  
In Section~\ref{sec:tailoredMILP}, we extend this result to the infinite discounted case by modifying the reward function. 

\section{MILP formulation for SMF policy}
\label{sec:tailoredMILP}

In this section, we detail how to use MILP~\eqref{pb:milp:MILP_pomdp} in order to solve the value function problem~\eqref{eq:memorylessFromBeliefValueFunction}, and the strength of its linear relaxation.
To do so, given an instance of a generic POMDP, we introduce a specific  instance of POMDP with finite horizon $T$ and the reward function $\tilder_t$ such that $\tilder_t(s,a) = \gamma^t r(s,a)$ for $t\leq T-1$, and $\tilder^{T}(s,a) = \gamma^T r(s,a) + \gamma^{T+1}\sum_{s'\in \calS}p(s'|s,a)v_{\rmMDP}(s')$, for any $s$ in $\calS$ and $a$ in $\calA$.
Given an initial belief $\bfb$ in $\calB$, we denote respectively by $\tildew_{\rmml}^T(\bfb)$ , $\tildez_{\rmRc}^T(\bfb)$ and $\tildez_{\rmR}^T(\bfb)$ the optimal value of memoryless policies, the value of the relaxation of our MILP with and without valid inequalities, for the instance with reward function $\tilder$.

\begin{prop}\label{prop:computingSMFApproximationWithFiniteHorizon}
    The optimal solutions of NLP~\eqref{pb:milp:NLP_pomdp} (resp.~MILP~\ref{pb:milp:MILP_pomdp}) coincide with the (resp. deterministic) SMF policies, and, $\hat v_{\rmSMF}^T(\bfb) = \tildew_{\rmml}^T(\bfb)$.
\end{prop}


\paragraph{Strengths of the linear relaxation}
It turns out that solving the linear relaxation of MILP~\eqref{pb:milp:MILP_pomdp} with reward function $\tilder$ gives a collection of upper bounds on the optimal history-dependent value $v^*(\bfb)$ for any belief $\bfb$ in $\calB$. Like Theorem~\ref{theo:milp:MDP_approx_equivalence_finiteHorizon} for finite horizon problem,  Theorem~\ref{theo:milp:MDP_approx_equivalence_infinite} below states that by adding valid inequalities~\eqref{eq:milp:Valid_cuts_pomdp} in MILP~\eqref{pb:milp:MILP_pomdp}, the linear relaxation gives also an upper bound. 
Theorem~\ref{theo:milp:MDP_approx_equivalence_infinite} says even more: the larger the finite horizon, the tighter the linear relaxations.  

\begin{restatable}{theo}{MDPequivalenceInfinitehorizon}\label{theo:milp:MDP_approx_equivalence_infinite}
    For any belief $\bfb$ in $\calB$, finite horizons $T, T'$ in $\bbZ_{+}$, the following inequalities hold:
    \begin{align}
        & v_{\rmSMF(T)}(\bfb) \leq v^*(\bfb) \leq \tilde z_{\rmRc}^{T'}(\bfb) \leq \tilde z_{\rmR}^{T'}(\bfb) \label{eq:milp:inequality_information}
        \\
    	& \tildez_{\rmRc}^{T+1}(\bfb) \leq \tildez_{\rmRc}^{T}(\bfb) \label{eq:milp:decreasing_bounds}
    \end{align}
    In addition, the upper bound $\tildez_{\rmR}^{T}(\bfb)$ is constant with respect to $T$, i.e., $\tilde z_{\rmR}^{T}(\bfb)= \tildez_{\rmR}^{0}(\bfb)$.
\end{restatable}

Inequality~\eqref{eq:milp:inequality_information} ensures that we can measure the quality the SMF policy.
Indeed, it says that the gap $\tildez_{\rmRc}^{T'}(\bfb)- v_{\rmSMF(T)}(\bfb)$ bounds the gap $v^{*}(\bfb) - v_{\rmSMF(T)}(\bfb)$. 
In addition, it tells us that we obtain a tighter measure of the optimality gap by adding valid inequalities~\eqref{eq:milp:Valid_cuts_pomdp} in the linear relaxation of MILP~\eqref{pb:milp:MILP_pomdp} with reward functions $\tilder$.
One can notice that Theorem~\ref{theo:milp:MDP_approx_equivalence_infinite} is a extended version of Proposition 4 of \citet{bertsimas2016decomposable} to POMDPs.
Inequality~\eqref{eq:milp:decreasing_bounds} ensures that increasing the finite horizon leads to a tighter upper bound. Intuitively, this result is due to the fact that by increasing the finite horizon, we induce more constraints, which correspond to conditional independence (see Appendix~\ref{app:validInequalities:probaInterpretation}). 
It highlights a trade-off between the quality of the relaxation and the tractability of such an upper bound. 

\section{Numerical experiments}
\label{sec:num}

In this section we provide numerical experiments showing the efficiency of our approach.
First, we describe the benchmark instances on which we evaluate our approach. 
Second, we show the performances of our MILP in solving memoryless POMDP with finite horizon and we compare the results with the optimal history-dependent value.
Third, we evaluate the performances of our SMF policy on the benchmark instances and we compare it with a the state-of-the-art POMDP algorithm.
All the mathematical programs have been written in \texttt{Julia} \citep{bezanson2017julia} with the \texttt{JuMP} \citep{DunningHuchetteLubin2017} interface and solved using \texttt{Gurobi} 9.0. \citep{gurobi} with the default settings. 
All the state-of-the-art POMDP solvers against which we compare our approach are implemented in the Julia library \texttt{POMDPs.jl} of \citet{EgorovSBWGK17}.
Experiments have been run on a server with 192Gb of RAM and 32 cores at 3.30GHz.

\subsection{Benchmark instances}
\label{sub:num:benchmarkInstances}

All the instances considered in this paper can be downloaded on~\url{http://pomdp.org/examples/}.
In Table~\ref{tab:schedulingResultsDetailed}, we summarize the main properties of these instances.
The first column indicates the name of the instances. 
Then, the next three columns report the sizes of the state spaces, observation spaces and action spaces.
Finally, the fourth column indicates the sparsity of the instances. Since there is no definition of this measure in the literature, we define it as the percentage of null entries in the transition probability distributions and the emission probability distributions. More precisely, we introduce
$$\mathrm{Sparsity} = \frac{\sum_{s,o} \mathds{1}_{p(o|s)= 0} + \sum_{s,a,s'}\mathds{1}_{p(s'|s,a)= 0}}{\vert \calS \vert \vert \calO \vert + \vert \calS \vert^2 \vert \calA \vert}.$$
The sparser the instances, the sparser the constraints matrix of MILP~\eqref{pb:milp:MILP_pomdp}. 
Even if this indicator has to be considered cautiously, it measures our ability to solve MILP~\eqref{pb:milp:MILP_pomdp}.

For most of the instances, an initial belief is provided. If no initial belief is provided, then we consider that the initial belief is an uniform distribution over the state space. In order to lighten the notation, we omit the dependence in the initial belief in the rest of the paper.

\begin{table}
\begin{minipage}{0.5\textwidth}
\begin{center}
\scalebox{0.6}
{

\begin{tabular}{ccc@{\hspace{0.2cm}}rrrr@{\hspace{0.2cm}}rrrrr}

\toprule

Name &&

\multicolumn{3}{c}{Size} && 

\multicolumn{1}{c}{Sparsity $(\%)$} \\


&&

$\calX_S$&

$\calX_O$&

$\calX_A$&

 & &  

\\

\midrule

\texttt{1d.noisy} && 4 & 2 & 2 && 45.8 \\
\texttt{1d} && 4 & 2 & 2 && 58.3 \\
\texttt{bridge-repair} && 5 & 5 & 12 && 47.8 \\
\texttt{cheese} && 11 & 7 & 4 && 84.3 \\

\texttt{cheng.D5-1} && 5 & 3 & 3 && 0.0  \\
\texttt{easy-bridge} && 5 & 5 & 2 && 59.0  \\ 
\texttt{ejs1} && 3 & 2 & 4 && 46.7 \\
\texttt{ejs2} && 2 & 2 & 2 && 0.0  \\
\texttt{ejs3} && 2 & 2 & 2 && 0.0 \\

\texttt{line4-2goals} && 4 & 1 & 2 && 40.0  \\
\texttt{marking} && 9 & 3 & 4 && 79.9 \\
\texttt{mcc-example1} && 4 & 3 & 3 && 42.9 \\

\texttt{paint} && 4 & 2 & 4 && 52.1 \\






\bottomrule


\end{tabular} 
}

\end{center}
\end{minipage}
\begin{minipage}{0.5\textwidth}
\begin{center}

\scalebox{0.6}
{

\begin{tabular}{ccc@{\hspace{0.2cm}}rrrr@{\hspace{0.2cm}}rrrrr}

\toprule

Name &&

\multicolumn{3}{c}{Size} && 

\multicolumn{1}{c}{Sparsity $(\%)$} \\


&&

$\calX_S$&

$\calX_O$&

$\calX_A$&

 & &  

\\

\midrule


\texttt{query.s2} && 9 & 3 & 2 && 5.6 \\
\texttt{saci-s12-a6-z5.95} && 12 & 5 & 6 && 70.3 \\
\texttt{sunysb} && 300 & 28 & 4 && 91.1 \\


\texttt{web-ad} && 4 & 5 & 3 && 38.9 \\
\texttt{web-mall} && 2 & 2 & 3 && 8.3 \\
\texttt{4x3} && 11 & 6 & 4 && 71.7 \\
\texttt{4x4} && 16 & 2 & 4 && 84.0 \\
\texttt{4x5x2} && 39 & 4 & 4 && 93.1 \\


\texttt{stand-tiger} && 4 & 4 & 4 && 59.4 \\
\texttt{mini-hall2} && 13 & 9 & 3 && 87.1 \\
\texttt{network} && 7 & 2 & 4 && 46.4 \\ 
\texttt{tiger} && 2 & 2 & 3 && 8.3 \\
\texttt{shuttle} && 8 & 5 & 3 && 79.5 \\

\bottomrule


\end{tabular} 
}
\end{center}
\end{minipage}

\caption{Description of the benchmark instances.}
\label{tab:schedulingResultsDetailed}

\end{table}

\subsection{Performances of memoryless policies on finite horizon}
\label{sub:num:memoryLessFiniteHorizon}

The goal of this section is to show how memoryless policies perform on several benchmark POMDP instances with finite horizon. To do so, given a finite horizon $T$ we compute the optimal memoryless value $w_{\rmml}^{T}$ by solving MILP~\eqref{pb:milp:MILP_pomdp} and we compare it with the optimal value of POMDP with finite horizon $w^{T}$: The smaller the difference $w^{T} - w_{\rmml}^{T}$, the better the optimal memoryless policies.
Since we cannot directly compute $v^{T}$, we compute instead its upper bound $z_{\rmRc}^{T}$. 
We introduce the gap $G(u) = (z_{\rmRc}^{T} - u)/z_{\rmRc}^{T}$ for any $u$ in $\bbR$.
Thanks to Theorem~\ref{theo:milp:MDP_approx_equivalence_infinite}, the gap $G(w_{\rmml}^{T})$ is a guarantee of the quality of the optimal memoryless policy.
In this section, we compute $w_{\rmml}^{T}$ in two ways: By solving MILP~\eqref{pb:milp:MILP_pomdp} and by solving NLP~\eqref{pb:milp:NLP_pomdp} using \texttt{Gurobi} NLP solver. 

Table~\ref{tab:num:memorylessPOMDP} summarizes the results. 
The first column indicates the instances. The next four columns report two solver statistics about the resolution of the mathematical programs: the optimality gap (Opt. gap) at the end of the resolution and the computation time in seconds. This optimality gap is the relative gap between the best solution obtained and the best bound computed during the resolution. 
The last two columns report the values of the gap $G(v_{\rmml}^{T})$ computed respectively using MILP~\eqref{pb:milp:MILP_pomdp} and NLP~\eqref{pb:milp:NLP_pomdp}. In order to distinguish these values, we denote respectively by $w_{\rmMILP}$ and $w_{\rmNLP}$ the values obtained by solving MILP~\eqref{pb:milp:MILP_pomdp} and NLP~\eqref{pb:milp:NLP_pomdp}.

All instances are solved with a finite horizon $T=20$. 
We set a computation time limit of $3600$ seconds. If the optimum has not been reached after $3600$ seconds, then we keep the best solution found and the best upper bound found at the end of the resolution.

\begin{table}
\centering
\scalebox{0.6}
{

\begin{tabular}{c@{\qquad}rr@{\qquad}rr@{\qquad\qquad}rr}

\toprule

\multirow{3}{*}{Instances} & \multicolumn{4}{c}{Math. prog. resolution} & \multicolumn{2}{c}{\text{Quality of ml. policy}}  \\
                           & \multicolumn{2}{c}{MILP} & \multicolumn{2}{c}{NLP} & \multicolumn{1}{c}{MILP} & \multicolumn{1}{c}{NLP} \\
                           & Opt. gap($\%$) & Time(s) & Opt. gap($\%$) & Time(s) &  $G(w_{\rmMILP}^{20})$ ($\%$) & $G(w_{\rmNLP}^{20})$ ($\%$) \\

\midrule

4x3  & Opt. &  137 &  Opt. & 564 & \multicolumn{1}{c}{26.2}  & \multicolumn{1}{c}{26.2}  \\

4x4  & Opt. &  14 & Opt. & 28 & \multicolumn{1}{c}{11.5}  & \multicolumn{1}{c}{11.5}  \\

4x5x2 & 0.2 & $>3600$ & 0.3 & $>3600 $ & \multicolumn{1}{c}{47.6}  & \multicolumn{1}{c}{47.6} \\

1d.noisy  & Opt. & 2.3 & Opt. & 2.3 & \multicolumn{1}{c}{17.5}  & \multicolumn{1}{c}{17.5} \\

1d  & Opt. & 0.5 & Opt. & 2.1 & \multicolumn{1}{c}{15.1}  & \multicolumn{1}{c}{15.1}  \\



cheese & Opt. & 8 & Opt. & 22 & \multicolumn{1}{c}{8.2}  & \multicolumn{1}{c}{8.2} \\

cheng.D3-1 & Opt. & 1888 & 0.1 & 3609.5 & \multicolumn{1}{c}{0.7}  & \multicolumn{1}{c}{0.7} \\

cheng.D4-1  & 0.1 & $>3600$ & 0.1 & $>3600$ & \multicolumn{1}{c}{3.5}  & \multicolumn{1}{c}{3.5} \\

cheng.D5-1  & 0.1 & $>3600$ & 0.1 & $>3600$ & \multicolumn{1}{c}{1.7}  & \multicolumn{1}{c}{1.7} \\

ejs1 & Opt. & 0.8 & Opt. & 2.8 & \multicolumn{1}{c}{6.8} & \multicolumn{1}{c}{6.8} \\

ejs2 & Opt. & 7.3 & Opt. & 11.7 & \multicolumn{1}{c}{7.2} & \multicolumn{1}{c}{7.2} \\

ejs3 & Opt. & 0.1 & Opt. & 1.8 & \multicolumn{1}{c}{37.6} & \multicolumn{1}{c}{37.6} \\


marking & Opt. & $57$ & Opt. & $66$  & \multicolumn{1}{c}{13.6}  & \multicolumn{1}{c}{13.6} \\

line4-2goals & Opt. & 0.1 & Opt. & 1.7 & \multicolumn{1}{c}{5.6}  & \multicolumn{1}{c}{5.6} \\

mcc-example1 & Opt. & 2.7 & Opt. & 2.9  & \multicolumn{1}{c}{6.4}  & \multicolumn{1}{c}{6.4} \\

mcc-example2 & Opt. & 2.8 & Opt. & 2.9  & \multicolumn{1}{c}{5.8}  & \multicolumn{1}{c}{5.8} \\

mini-hall2  & Opt. & 499 & Opt. & 339 & \multicolumn{1}{c}{15.4}  & \multicolumn{1}{c}{15.4} \\

paint & Opt. & 3e-1 & Opt. & 4e-1 & \multicolumn{1}{c}{63.7}  & \multicolumn{1}{c}{63.8} \\

parr95  & Opt. & 20 & Opt. & 71 & \multicolumn{1}{c}{19.0} & \multicolumn{1}{c}{19.0} \\

query.s2  & Opt. & 1.7 & Opt. & 12 & \multicolumn{1}{c}{0.1} & \multicolumn{1}{c}{0.1} \\


saci-s12-a6-z5.95 & Opt. & 129 & Opt. & 412 & \multicolumn{1}{c}{0.0}  & \multicolumn{1}{c}{0.0} \\

shuttle & Opt. & 4 & Opt. & 11 & \multicolumn{1}{c}{7.7}  & \multicolumn{1}{c}{7.7} \\

stand-tiger & 0.7 & $>3600$ & 0.7 & $>3600$ & \multicolumn{1}{c}{80.9}  & \multicolumn{1}{c}{78.8} \\

tiger & Opt. & 945 & Opt. & 2530 & \multicolumn{1}{c}{122.2}  & \multicolumn{1}{c}{122.2} \\

network & Opt. & 1368 & Opt. & 2268 & \multicolumn{1}{c}{36.0}  & \multicolumn{1}{c}{36.0} \\

web-mall & Opt. & $1657$ & Opt. & $2888$ & \multicolumn{1}{c}{64.5}  & \multicolumn{1}{c}{64.5} \\

web-ad & Opt. & 1 & Opt. & $8$ & \multicolumn{1}{c}{0.0}  & \multicolumn{1}{c}{0.0} \\

aircraftID & Opt. & $22$ & Opt. & $202$ & \multicolumn{1}{c}{0.0}  & \multicolumn{1}{c}{0.0} \\

\bottomrule


\end{tabular} 
}
\caption{MILP performances and quality of memoryless policies on POMDPs with finite horizon.}
\label{tab:num:memorylessPOMDP}
\end{table}

\subsection{Performances of SMF policy on infinite discounted horizon}
\label{sub:num:infinite_horizon}

We evaluate the efficiency of our approach against SARSOP solver \citep{Kurniawati08sarsop}, which is one of the state-of-the-art POMDP solver.
For each instance, we set the discount factor $\beta = 0.95$.
A policy is evaluated by running $1000$ simulations over a large finite horizon of $100$ time steps. Then, we compute the average of the total reward over the $1000$ simulations.
The SMF policy is computed with two different rolling horizon $T$ in $\{2,5\}$. 
In addition, we compute the upper bounds $\tildez_{\rmRc}^{T_{\rmub}}$ with $T_{\rmub}=100$. Thanks to Theorem~\ref{theo:milp:MDP_approx_equivalence_infinite}, the value $\tildez_{\rmRc}^{T_{\rmub}}$ is an upper bound of $v^*$ for any value of $T_{\rmub}$. Hence, the gap $G^{T_{\rmub}} = (\tildez_{\rmRc}^{T_{\rmub}} - v_{\rmSMF(T)})/\tildez_{\rmRc}^{T_{\rmub}}$ is an indicator of the quality of the policy: The smaller the gap, the better the policy.
However we do not know how far is the value $\tildez_{\rmRc}^{T_{\rmub}}$ from the optimal value $v^*$, so the value of this gap can be large, even if the policy is of good quality.

In order to illustrate Theorem~\ref{theo:milp:MDP_approx_equivalence_infinite}, we also compute the upper bounds with finite horizon $T_{\rmub}$ in $\{20, 100 \}$ and we report the gap $G_{\rmub}^{T_{\rmub}} = (\tildez_{\rmR}^{0} - \tildez_{\rmRc}^{T_{\rmub}})/\tildez_{\rmR}^{0}$, which indicates the relative improvement of the upper bound with respect to the largest upper bound, i.e., the MDP approximation.

Since SARSOP policy is computed offline and our policy is computed online, we cannot easily compare the computation times. 
We propose to report the averaged time to take an action. The SARSOP computation time reported in Table~\ref{tab:num:FSMBSPolicyInfiniteDisc} correspond to the offline computation time divided by the number of simulations $1000$.  

Table~\ref{tab:num:FSMBSPolicyInfiniteDisc} summarizes the results. The first column indicates the POMDP instance. 
The next four columns indicate respectively the gaps $G^{100}$ and the average computation time, each of them obtained by simulating SMF policy with rolling horizon $T_{\rmr}=2$ and $T_{\rmr}=5$.
Then, the last two columns indicate respectively the gaps $G^{100}$ and the average computation time obtained by simulating SARSOP policy.
Finally, the last two columns report the upper bound improvement, which is indicated by the gap $G_{\rmub}^{T_{\rmub}}$ with $T_{\rmub}$ in $\{20,100\}$.

One can observe that in most instances, our SMF policy performs at least as good as SARSOP policy.  
As expected from Theorem~\ref{theo:milp:MDP_approx_equivalence_infinite}, we observe that $G_{\rmub}^{100} \geq G_{\rmub}^{20}$ for every instances. We could have considered larger horizon $T$ to compute tighter upper bound $\tildez_{\rmRc}^{T}$, but for the majority of instances, there is not much improvement for $T_{\rmub} > 100$.
In addition, we observe that by increasing the time horizon $T$ for policy $\bfdelta_{\rmSMF(T)}$, we obtain a higher expected reward. This was expected because the policy with a larger finite horizon considers mores steps in the future, which makes it more anticipative.

Table~\ref{tab:num:FSMBSPolicyInfiniteDisc} shows also that the average required computation time to decide the next action to take according to the SMF policy is never above a half second.  There is a large class of applied problems where this computation time seems reasonable (maintenance, robot navigation, etc.) 

To conclude, we observe that the value of SMF policy is in the worst case $20 \%$ below the optimal value $v^*$ for $60 \%$ of the instances in Table~\ref{tab:num:FSMBSPolicyInfiniteDisc}. It shows the quality of SMF policy on a large class of instances despite the limitations of the memoryless policies used in its computation. 


\begin{table}
\centering
\scalebox{0.7}
{
\begin{tabular}{c@{\qquad}rr@{\qquad\qquad}rr@{\qquad\qquad}rr@{\qquad\qquad}rr}
\toprule

Instances & \multicolumn{2}{p{2cm}}{SMF policy $T_{\rmr}=2$} & \multicolumn{2}{p{2cm}}{SMF policy $T_{\rmr}=5$} &  \multicolumn{2}{p{2cm}}{SARSOP} &  \multicolumn{2}{p{2cm}}{Upper bounds} \\
 
& $G^{100}$ & Time & $G^{100}$ & Time & $G^{100}$ & Time & $G_{\mathrm{ub}}^{100}$  & $G_{\mathrm{ub}}^{20}$   \\

\midrule

\texttt{4x3} & 98.5  & 0.04 & \textbf{98.5} & 0.1 & 98.6 & 0.002  & 8.2 & 4.9 \\
\texttt{4x4} & 62.2  & 0.03 & 60.6 & 0.07 & \textbf{58.9} & 0.002  & 4.3 & 2.3 \\
\texttt{4x5x2} & \textbf{37.8} & 0.09 & 39.7 & 0.6 & \textbf{37.8} & 0.002  & 14.5 & 8.5 \\
\texttt{1d.noisy} & 15.7 & 0.01 & \textbf{15.7}  & 0.03 & \textbf{15.7} & 0.005  & 9.7 & 5.4 \\
\texttt{1d} & 17.7 & 0.01 & \textbf{17.7} & 0.02 & 20.0 & 0.002 & 11.7 & 7.2 \\
\texttt{cheese} & 89.3 & 0.04 & 89.3 & 0.1 & \textbf{83.4} & 0.002  & 98.6 & 98.7 \\
\texttt{cheng.D3-1} & 0.2 & 0.01 & \textbf{0.2} & 0.2 & \textbf{0.2} & 0.05  & 7.4 & 4.6 \\
\texttt{cheng.D4-1} & 2.7 & 0.02 & \textbf{2.7} & 0.6 & 2.8 & 3.8  & 12.1 & 7.7 \\
\texttt{cheng.D5-1} & 1.3 & 0.02 & \textbf{1.3} & 0.4 & \textbf{1.3} & 3.5  & 11.7 & 7.3 \\
\texttt{ejs1} & 27.5 & 0.01 & \textbf{26.9} & 0.02 & \textbf{26.9} & 0.002  & 1.4 & 0.8 \\
\texttt{ejs2} & 6.7 & 0.01 & \textbf{6.7} & 0.03 & 6.8 & 0.1 & 20.0 & 12.5 \\
\texttt{ejs3} & 16.4 & 0.01 & \textbf{16.4} & 0.03 & \textbf{16.4} & 0.002 & 0.0 & 0.0 \\
\texttt{mini-hall2} & 97.2 & 0.04 & \textbf{93.0} & 0.1 & \textbf{93.0} & 0.002 & 5.8 & 3.1 \\
\texttt{marking} & 4.5 & 0.03 & 4.5 & 0.05 & \textbf{1.9} & 0.002  & 1.5 & 1.5 \\
\texttt{line4-2goals} & 0.5 & 0.01 & \textbf{0.5} & 0.02 & \textbf{0.5} & 0.002  & 2.6 & 2.6 \\
\texttt{mcc-example1} & 1.5 & 0.02 & \textbf{1.4} & 0.06 & 1.5 & 0.003  & 3.9 & 2.6 \\
\texttt{mcc-example2} & 1.5 & 0.02 & \textbf{1.4} & 0.06 & 1.5 & 0.003  & 3.9 & 2.6 \\
\texttt{paint} & 100.0 & 0.02 & 60.0 & 0.2 & \textbf{56.6} & 0.002  & 42.2 & 24.7 \\
\texttt{parr95} & 16.8 & 0.03 & 14.8 & 0.04 & \textbf{3.3} & 0.002  & 38.4 & 22.8 \\
\texttt{query.s2} & Opt. & 0.03 & \textbf{Opt.} & 0.04 & \textbf{Opt.} & 0.002  & 0.3 & 0.2 \\
\texttt{saci-s12-a6-z5.95} & 27.4 & 0.05 & \textbf{27.4} & 0.3 & \textbf{27.4} & 0.002  & 0.0 & 0.0 \\
\texttt{shuttle} & 67.4 & 0.02 & \textbf{67.4} & 0.05 & 202.3 & 0.001  & 0.0 & 0.0 \\
\texttt{stand-tiger} & 100.0 & 0.02 & 100.0 & 0.05 & \textbf{81.7} & 0.002  & 0.0 & 0.0 \\
\texttt{tiger} & 81.8 & 0.01 & 81.8 & 0.08 & \textbf{78.8} & 0.002 & 53.5 & 32.5 \\
\texttt{network} & \textbf{14.7} & 0.02 & 17.9 & 0.2 & 15.0 & 3.6  & 22.8 & 13.7 \\
\texttt{web-mall} & 59.4 & 0.01 & 59.4 & 0.1 & \textbf{59.1} & 0.003  & 51.6 & 31.4 \\
\texttt{web-ad} & Opt. & 0.02 & \textbf{Opt.} & 0.04 & \textbf{Opt.} & 0.002  & 0.0 & 0.0 \\
\texttt{aircraftID} & 6.0 & 0.05 & \textbf{6.0} & 0.3 & \textbf{6.0} & 0.002  & 0.0 & 0.0 \\

\bottomrule

\end{tabular} 
}
\caption{SMF policy performances on POMDP with discounted infinite horizon.}
\label{tab:num:FSMBSPolicyInfiniteDisc}
\end{table}

\section*{Acknowledgments}

We are grateful to Prof. Fr\'ed\'eric Meunier for his helpful remarks.
The authors gratefully acknowledge the financial support of the Operations Research and Machine Learning chair between Ecole des Ponts Paristech and Air France.

\newpage
\appendix


\section{Upper bound on the optimal history-dependent value}
\label{app:validInequalities}

In this appendix, we give a probabilistic interpretation of the valid inequalities of Section~\ref{sub:milp:valid_cuts} and the proofs of the results of Section~\ref{sec:milp}.

\subsection{Probabilistic interpretation of the valid inequalities}
\label{app:validInequalities:probaInterpretation}

We show in this section that the valid inequalities introduced in Section~\ref{sub:milp:valid_cuts} can be interpreted in terms of conditional independences that are stronger than the ones induced in NLP~\eqref{pb:milp:NLP_pomdp}.

Given a feasible solution $(\bfmu,\bfdelta)$ of the linear relaxation of~\eqref{pb:milp:MILP_pomdp}, $\bfmu$ can still be interpreted as the vector of moments of a probability distribution $\bbQ_{\bfmu}$ over 
$\left(\calS \times \calO \times \calA\right)^T \times \calS$.
However, as it has been mentioned above, the vector $\bfmu$ does not necessarily correspond to the vector of moments of $\bbP_{\bfdelta}$, which is due to the fact that $(\bfmu,\bfdelta)$ does not necessarily satisfy the nonlinear constraints~\eqref{eq:milp:NLP_indep_action}.
Besides, constraints~\eqref{eq:milp:NLP_indep_action} is equivalent to the property that, 
\begin{equation}\label{eq:milp:strongIndep}
\text{according to $\bbQ_{\bfmu}$, action $A_t$ is independent from state $S_t$ given observation $O_t$.} 
\end{equation}
Hence, given a feasible solution $(\bfmu,\bfdelta)$ of the linear relaxation of MILP~\eqref{pb:milp:MILP_pomdp}, the distribution $\bbQ_{\bfmu}$ does not necessarily satisfy the conditional independences~\eqref{eq:milp:strongIndep}.
Remark that~\eqref{eq:milp:strongIndep} implies the weaker result that,
\begin{equation}\label{eq:milp:weakIndep}
    \text{according to $\bbQ_{\bfmu}$, $A_t$ is independent from $S_t$ given $O_t$, $A_{t-1}$ and $S_{t-1}$.}
\end{equation}
Proposition~\ref{prop:milp:valid_cuts_pomdp} says that the independences in~\eqref{eq:milp:weakIndep} are not satisfied in general by a feasible solution $(\bfmu,\bfdelta)$ of the linear relaxation of MILP~\eqref{pb:milp:MILP_pomdp}, but that we can enforce them using linear inequalities~\eqref{eq:milp:Valid_cuts_pomdp} on $(\bfmu,\bfdelta)$ in an extended formulation.

\subsection{Proofs of Sections~\ref{sec:milp}}
\label{app:validInequalities:proofs}

In this section, we give the proofs of Theorems~\ref{theo:milp:NLP_optimal_solution},~\ref{theo:milp:MDP_approx_equivalence_infinite} an Proposition~\ref{prop:milp:valid_cuts_pomdp}.
In all the proofs, we omit the conditioning of the probabilities on the event ${S_0 \sim \bfb}$.

\NLPoptimalSolution*

\begin{proof}[Proof of Theorem~\ref{theo:milp:NLP_optimal_solution}]
    Let $(\bfmu, \bfdelta)$ be a feasible solution of Problem~\eqref{pb:milp:NLP_pomdp}. 
    We prove by induction on $t$ that $\mu_s^0 = \bbP_{\bfdelta}\big(S_0 = s\big)$, $\mu_{soa}^t=\bbP_{\bfdelta}\big(S_t = s, O_t=o, A_t=a \big)$ and $\mu_{sa}^t =\bbP_{\bfdelta}\big(S_t = s, A_t=a \big)$ for $t$ in $[T]$. 

    First, note that at time $t=0$, Constraints~\eqref{eq:milp:NLP_initial_policy} ensure that $\delta_{a|o}^0$ does not depend on $o$, i.e., $\delta_{a|o}^0 = \bbP_{\bfdelta}(A_0=a|O_0=o) = \bbP_{\bfdelta}(A_0=a)$.
    At time $t=0$, the statement is true because
    \begin{align*}
        & \mu_s^0 = \bfb(s) = \bbP_{\bfdelta}\left(S_0 =s\right) & \\
        & \mu_{soa}^0 = \delta^0_{a|o} p(o|s) \mu_s^0 =  \bbP_{\bfdelta}\left(A_0=a | O_0 = o\right) \bbP_{\bfdelta}\left(O_0=o | S_0 =s\right) \bbP_{\bfdelta}\left(S_0 =s\right) \\
        & \mu_{sa}^0 = \sum_{o \in \calO} \mu_{soa}^0 = \sum_{o \in \calO} \bbP_{\bfdelta}\left(S_0=s,O_0=o,A_0=a\right) = \bbP_{\bfdelta}\left(S_0=s,A_0=a\right) 
    \end{align*}
    Suppose that the induction hypothesis holds up to time $t-1$. 
    Then, Constraints~\eqref{eq:milp:NLP_consistency_s} and the induction hypothesis ensure that:
    \begin{align*}
        \mu_s^{t} &= \sum_{s' \in \calS, a' \in \calA} p(s|s',a')\mu_{s'a'}^{t-1} = \sum_{s' \in \calS, a' \in \calA} p(s|s',a') \bbP_{\bfdelta}\left(S_{t-1}=s',A_{t-1}=a' \right) \\
                    & = \sum_{s' \in \calS, a' \in \calA} \bbP_{\bfdelta}\left(S_{t-1}=s',A_{t-1}=a',S_{t}=s\right)\\
                    &= \bbP_{\bfdelta} \left(S_{t} = s\right),
    \end{align*}
    which shows that~\eqref{eq:milp:def_moments_s} holds.

    By combining Constraints~\eqref{eq:milp:NLP_indep_action} and the induction hypothesis, we obtain that 
    \begin{align*}
    \mu_{soa}^{t} &= \delta_{a|o}^t p(o|s) \mu_{s}^t \\
    &= \bbP_{\bfdelta}\left(A_t=a | O_t =o\right) \bbP_{\bfdelta}\left(O_t=o | S_t =s\right) \bbP_{\bfdelta}\left(S_t =s\right) \\
    & = \delta_{a|o}^t p(o|s) \bbP_{\bfdelta}\left(S_t=s \right) \\
    &= \bbP_{\bfdelta}\left(S_t=s, O_t=o, A_t=a \right),
    \end{align*}
    where the last equality comes from the conditional independence and the law of total probability.
    It shows that~\eqref{eq:milp:def_moments_soa} holds. 
    Finally, Constraints~\eqref{eq:milp:NLP_consistency_sa} ensure that
    \begin{align*}
        \mu_{sa}^t = \sum_{o \in \calO} \mu_{soa}^t 
        &= \sum_{o \in \calO} \bbP_{\bfdelta}^t \left( S_t=s,O_t=o,A_t=a\right) \\
        &= \bbP_{\bfdelta}^t \left( S_t=s,A_t=a\right),
    \end{align*}
    which shows that~\eqref{eq:milp:def_moments_sa} holds.
    Consequently, for any feasible solution $(\bfmu,\bfdelta)$ of Problem~\eqref{pb:milp:NLP_pomdp},
    \begin{align*}
        \sum_{t=0}^T \sum_{\substack{s,s' \in \calS \\ a \in \calA}} r_t(s,a) \mu_{sa}^t
        = \sum_{t=0}^T \sum_{\substack{s,s' \in \calS \\ a \in \calA}} r_t(s,a) \bbP_{\bfdelta}\left(S_t=s,A_t=a\right) &= \bbE_{\bfdelta} \bigg[ \sum_{t=0}^{T}r_t(S_t,A_t) | \bfb \bigg] \\
                        &= \sum_{a \in \calA} \delta_{a|o}^0 \bbE_{\bfdelta} \bigg[ \sum_{t=0}^{T}r_t(S_t,A_t) | \bfb, a \bigg] 
    \end{align*}
    for any element $o$ in $\calO$.
    Hence, it follows that $(\bfmu,\bfdelta)$ is an optimal solution of NLP~\eqref{pb:milp:NLP_pomdp} if, and only if:
    \begin{align*}
        &\bfdelta \in \argmax \sum_{a \in \calA} \delta_{a|o}^0 \bbE_{\bfdelta} \bigg[ \sum_{t=0}^{T}r_t(S_t,A_t) | \bfb, a \bigg]  \iff  \bfdelta \in \argmax_{\bfdelta \in \Deltaml^T, a \in \calA} \bbE_{\bfdelta} \bigg[ \sum_{t=0}^{T}r_t(S_t,A_t) | \bfb, a \bigg],
    \end{align*}
    which means that $\bfdelta$ is an optimal solution of Problem~\eqref{pb:POMDPmlfiniteHorizon}
    Consequently, $w_{\rmml}^{T}(\bfb) = z_{\rmml}^{T}(\bfb)$. It achieves the proof.
\end{proof}


\ValidInequalities*

\begin{proof}[Proof of Proposition~\ref{prop:milp:valid_cuts_pomdp}]
    Let $(\bfmu,\bfdelta)$ be a feasible solution of Problem~\eqref{pb:milp:MILP_pomdp}. We define
    $$\mu_{s'a'soa}^t = \delta^t_{a|o}p(o|s)p(s|s',a')\mu_{s'a'}^{t-1}$$ for all $(s',a',s,o,a) \in \calS \times \calA \times \calS \times \calO \times \calA$, $t \in [T]$. 
    These new variables satisfy constraints in \eqref{eq:milp:Valid_cuts_pomdp} :
    \begin{align*}
        \sum_{a \in \calA} \mu_{s'a'soa}^t 
        &= \left(\sum_{a \in \calA} \delta^t_{a|o}\right)p(o|s)p(s|s',a')\mu_{s'a'}^{t-1} = p(o|s)p(s|s',a')\mu_{s'a'}^{t-1}\\
        \sum_{a' \in \calA, s' \in \calS} \mu_{s'a'soa}^t &= \left(\sum_{a' \in \calA, s' \in \calS} p(s|s',a')\mu_{s'a'}^{t-1}\right) \delta^t_{a|o} p(o|s) = \delta^t_{a|o} p(o|s) \mu_{s}^t
    \end{align*}

    \noindent The remaining constraint \eqref{eq:milp:Valid_cuts_pomdp_main} is obtained using the following observation :
    \begin{align*}
        \frac{\mu_{s'a'soa}^t}{\sum_{s'' \in \calS} \mu_{s'a's''oa}^t} 
        &= \frac{p(o|s)p(s|s',a')\mu_{s'a'}^{t-1}}{\sum_{s'' \in \calS} p(o|s'')p(s''|s',a')\mu_{s'a'}^{t-1}} =  \frac{p(o|s)p(s|s',a')\sum_{\ovs} \mu_{s'a'\ovs}^{t-1}}{\sum_{s'' \in \calS} p(o|s'')p(s''|s',a')\sum_{\ovs} \mu_{s'a'\ovs}^{t-1}} \\ 
        &=  \frac{\displaystyle p(o|s)p(s|s',a')}{\displaystyle \sum_{s'' \in \calS} p(o|s'')p(s''|s',a')}
    \end{align*}
    By setting $p(s|s',a',o) = \frac{\displaystyle p(o|s)p(s|s',a')}{\displaystyle \sum_{\overline{s} \in \calS} p(o|\overline{s})p(\overline{s}|s',a')}$, equality \eqref{eq:milp:Valid_cuts_pomdp_main} holds. If $\sum_{s'' \in \calS} \mu_{s'a's''oa}^t = 0$, then $\mu_{s'a'soa}^t=0$ and constraint~\eqref{eq:milp:Valid_cuts_pomdp_main} is satisfied.

    Now we prove that there exists a solution $\bfmu$ of the linear relaxation of MILP~\eqref{pb:milp:MILP_pomdp} that does not satisfy inequalities~\eqref{eq:milp:Valid_cuts_pomdp}. We define such a solution $(\bfmu,\bfdelta)$:
    \begin{align}
        &\mu_s^0 = \bfb(s) \label{eq:milp:proof_mu_s_init} \\
        &\mu_{soa}^t = \begin{cases}
                                &p(o|s)\mu_s^t, \ \text{if} \ a = \phi(s) \label{eq:milp:proof_mu_soa}\\
                                &0,\  \text{otherwise} 
                            \end{cases}, & \text{if \ $t\geq 0$}, \\
        &\mu_{sa}^{t} =  \sum_{o \in \calO} \mu_{soa}^t & \label{eq:milp:proof_mu_sa}\\
        &\mu_{s}^t = \sum_{s'\in \calS, a' \in \calA} p(s|s',a')\mu_{s'a'}^{t-1}, & \text{if} \ t \geq 1, \label{eq:milp:proof_mu_s} \\
        &\delta_{a|o}^t = \begin{cases}
                                &\frac{\sum_{s \in \calS}\mu_{soa}^t}{\sum_{s \in \calS, a \in \calA} \mu_{soa}^{t}} \ \text{if} \ \sum_{s \in \calS, a \in \calA} \mu_{soa}^{t} \neq 0 \label{eq:milp:proof_delta}\\
                                &\mathds{1}_{\tilde{a}}(a),\  \text{otherwise} 
                            \end{cases} &
    \end{align}
    where $\phi : \calS \rightarrow \calA$ is an arbitrary mapping and $\tilde{a}$ is an arbitrary element in $\calA$. We prove that $\bfmu$ is a feasible solution of the linear relaxation of MILP~\eqref{pb:milp:MILP_pomdp}. 

    First, Constraints~\eqref{eq:milp:NLP_state_initial}-\eqref{eq:milp:NLP_consistency_s} are satisfied because of the definition $\bfmu$. Indeed, Equalities~\eqref{eq:milp:proof_mu_s_init}-\eqref{eq:milp:proof_mu_s} exactly correspond to Constraints~\eqref{eq:milp:NLP_state_initial}-\eqref{eq:milp:NLP_consistency_s}.

    Second, it remains to prove that constraints ~\eqref{eq:milp:MILP_McCormick_1},~\eqref{eq:milp:MILP_McCormick_2}, ~\eqref{eq:milp:MILP_McCormick_3} are satisfied.
    Inequality~\eqref{eq:milp:MILP_McCormick_1} is satisfied because
    \begin{align*}
        \mu_{soa}^t \leq \max\left(0,p(o|s)\mu_{s}^{t}\right) \leq p(o|s) \mu_{s}^{t},
    \end{align*}
    Inequality~\eqref{eq:milp:MILP_McCormick_2} is satisfied because
    \begin{align*}
        \mu_{soa}^t \leq \sum_{s' \in \calS}\mu_{s'oa}^t &= \delta_{a|o}^t \sum_{s' \in \calS,a' \in \calA} \mu_{s'oa'}^t \\
                        &= \delta_{a|o}^t \sum_{s' \in \calS,a' \in \calA} \mathds{1}_{\phi(s)}(a') p(o|s')\mu_{s'}^t \\
                        &= \delta_{a|o}^t \underbrace{\sum_{s' \in \calS}p(o|s')\mu_{s'}^t}_{\leq 1} \leq \delta_{a|o}^t 
    \end{align*}
    where we used definition~\eqref{eq:milp:proof_delta} for the first equality.
    Inequality, \eqref{eq:milp:MILP_McCormick_3} is satisfied because
    \begin{align*}
        \mu_{soa}^t - p(o|s) \mu_{s}^{t} & \geq \sum_{s' \in \calS} \overbrace{\mu_{s'oa}^t - p(o|s') \mu_{s'}^{t}}^{\leq 0} \\
        &=  \delta_{a|o}^t \sum_{s' \in \calS}  p(o|s') \mu_{s'}^t - \sum_{s' \in \calS}  p(o|s') \mu_{s'}^{t}   \\
        &= \sum_{s' \in \calS}  p(o|s') \mu_{s'}^t (\delta_{a|o}^t - 1) \\
        &\geq \delta_{a|o}^t - 1,
    \end{align*}
    which yields ~\eqref{eq:milp:MILP_McCormick_3}.
    Therefore, $(\bfmu,\bfdelta)$ is a solution of the linear relaxation of MILP~\eqref{pb:milp:MILP_pomdp}.

    Now, we prove that such a solution does not satisfy inequalities~\eqref{eq:milp:Valid_cuts_pomdp}. We define the new variables:
    \begin{align*}
        \mu_{s'a'soa}^t = \begin{cases}
                            & \mu_{s'a'}^{t-1} p(s|s',a')\frac{\mu_{soa}^t}{\sum_{o'\in \calO,a' \in \calA}\mu_{so'a'}^t} \ \text{if} \ \sum_{o\in \calO,a \in \calA}\mu_{soa}^t \neq 0 \\
                            & 0 \ \text{otherwise}
                            \end{cases}
    \end{align*}
    Hence, $\bfmu$ satisfies constraints~\eqref{eq:milp:Valid_cuts_pomdp_consistency1} and \eqref{eq:milp:Valid_cuts_pomdp_consistency2}. However, constraint~\eqref{eq:milp:Valid_cuts_pomdp_main} is not satisfied in general. 
    Indeed, since the mapping $\phi$ is arbitrary, we can set $\phi$ such that $p(s|s',a',o) >0$ and $\mu_{s'a'soa}^t = 0$.
    Therefore, there exists a solution $\bfmu$ of the linear relaxation of MILP~\eqref{pb:milp:MILP_pomdp} that does not satisfy inequalities~\eqref{eq:milp:Valid_cuts_pomdp}. It achieves the proof.
\end{proof}

\MDPequivalence*

\begin{proof}[Proof of Theorem~\ref{theo:milp:MDP_approx_equivalence_finiteHorizon}]
    Let $\bfb$ in $\calB$ and $T$ in $\bbZ_{+}$.
    First, we prove the equivalence between the linear relaxation of our MILP~\eqref{pb:milp:MILP_pomdp} and the MDP approximation with finite horizon $T$. To do so, we use the dual formulation of~\citet{Epenoux1963} that solves MDP problem on finite horizon
    \begin{subequations}\label{pb:app_proof:LP_mdp}
        \begin{alignat}{2}
            \max_{\bfmu, \bfdelta}  \enskip & \sum_{t=0}^T \sum_{\substack{s \in \calS,a \in \calA}} r_t(s,a) \mu_{sa}^t  & \quad & \label{eq:app_proof:LP_obj_function}\\
            \mathrm{s.t.} \enskip
             & \mu_{s}^{0} = \bfb(s) &  \quad  \forall s \in \calS \label{eq:app_proof:LP_state_initial} \\
             & \sum_{a \in \calA} \mu_{sa}^t = \mu_s^{t} &  \quad  \forall s \in \calS, t \in [T] \label{eq:app_proof:LP_consistency_s} \\
             & \mu_s^{t+1} = \sum_{s' \in \calS, a' \in \calA} p(s|s',a')\mu_{s'a'}^t &  \quad  \forall s \in \calS, t \in [T-1] \label{eq:app_proof:LP_consistency_s} \\
             & \bfmu \geq 0 \label{eq:app_proof:LP_constraint_policy} 
        \end{alignat}
    \end{subequations}
    Hence, it suffices to prove that linear relaxation of MILP~\eqref{pb:milp:MILP_pomdp} is equivalent to Linear program~\eqref{pb:app_proof:LP_mdp}.
    Note that the two objective functions are the same. Hence, we only need to prove that we can construct a feasible solution from a problem to another.
    
    Let $(\bfmu,\bfdelta)$ be a feasible solution of the linear relaxation of MILP~\eqref{pb:milp:MILP_pomdp}. 
    Constraints~\eqref{eq:milp:NLP_state_initial}-\eqref{eq:milp:NLP_consistency_s} ensure that $(\mu_s^1,\mu_{sa}^t)_{t \in [T]}$ is a feasible solution of Linear program~\eqref{pb:app_proof:LP_mdp}.

    Let $\bfmu$ be a feasible solution of Linear program~\eqref{pb:app_proof:LP_mdp}. It suffices to define variables $\delta_{a|o}^t$ and $\mu_{soa}^t$ for all $a$ in $\calA$, $o$ in $\calO$, $s$ in $\calS$, and $t$ in $[T]$. We define these variables using ~\eqref{eq:milp:proof_mu_soa} and~\eqref{eq:milp:proof_delta}. 
    In the proof of Proposition~\ref{prop:milp:valid_cuts_pomdp}, we proved that $(\bfmu,\bfdelta)$ is a feasible solution of the linear relaxation of MILP~\eqref{pb:milp:MILP_pomdp}.  
    Consequently, the equivalence holds and $z_{\rmR}^{T}(\bfb) = v_{\rmMDP}^{T}(\bfb)$.

    Now we prove that inequalities~\eqref{eq:milp:inequality_information} hold.
    Note that Proposition~\eqref{prop:milp:valid_cuts_pomdp} ensures that
    $$w_{\rmml}^{T}(\bfb) \leq z_{\rmRc}^{T}(\bfb) \leq z_{\rmR}^{T}(\bfb).$$
    It remains to prove the two following inequalities.
    \begin{align}
        &w_{\rmml}^{T}(\bfb) \leq v^{T}(\bfb) \label{eq:milp:ineq1}\\
        &v^{T}(\bfb) \leq z_{\rmRc}^{T}(\bfb) \label{eq:milp:ineq2}
    \end{align}

    First, we prove Inequality~\eqref{eq:milp:ineq1}. By definition, we have $\Deltaml^T \subseteq \Deltahis^T$.
    Hence, we obtain $w_{\rmml}^{T}(\bfb) \leq v^{T}(\bfb)$. Theorem~\ref{theo:milp:NLP_optimal_solution} ensures that $w_{\rmml}^{T}(\bfb) \leq v^{T}(\bfb)$.
    Therefore the inequality $w_{\rmml}^{T}(\bfb) \leq v^{T}(\bfb) \leq z_{\rmR}^{T}(\bfb)$ holds.

    Now we prove Inequality~\eqref{eq:milp:ineq2}.
    The proof is based on a probabilistic interpretation of the valid inequalities~\eqref{eq:milp:Valid_cuts_pomdp}. It suffices to proves that for any policy $\bfdelta$ in $\Deltahis^T$, the probability distribution $\bbP_{\bfdelta}$ satisfies the weak conditional independences~\eqref{eq:milp:weakIndep}.
    Let $\bfdelta \in \Deltahis^T$. The probability distribution $\bbP_{\bfdelta}$ over the random variables $(S_t,A_t,O_t)_{0\leq t \leq T}$ according to $\bfdelta$ is exactly
    \begin{align}\label{eq:milp:proof_distrib}
        \bbP_{\bfdelta} (\left(S_t=s_t,O_t=o_t,A_t=a_t \right)_{0\leq t \leq T}) &= \bbP_{\bfdelta}(S_0=s_0)\prod_{t=0}^T \bbP_{\bfdelta}(S_{t+1}=s_{t+1}|S_t=s_t,A_t=a_t) \notag \\
        &\bbP_{\bfdelta}(O_t=o_t|S_t=s_t) \delta^t_{a_t|h_t}
    \end{align}
    where $h_t = \{O_0=o_0,A_0=a_0,O_2=o_2,\ldots, O_t=o_t\}$ is the history of observations and actions. Note that the policy at time $t$ is the conditional probability $\delta^t_{a_t|h_t} = \bbP_{\bfdelta}(A_t=a_t|H_t=h_t)$.
    We define:
    \begin{align*}
        &\mu_{s}^0 = \bbP_{\bfdelta}(S_0=s)\\
        &\mu_{sa}^t = \bbP_{\bfdelta}(S_t=s,A_t=a)\\
        &\mu_{soa}^t = \bbP_{\bfdelta}(S_t=s,O_t=o,A_t=a)\\
        &\mu_{s'a'soa}^t = \bbP_{\bfdelta}(S_{t-1}=s',A_{t-1}=a',S_t=s,O_t=o,A_t=a)
    \end{align*}
    We define the policy $\tilde{\bfdelta}$ using~\eqref{eq:milp:proof_delta}.
    
    We have already proved that Constraints~\eqref{eq:milp:NLP_state_initial}-\eqref{eq:milp:NLP_consistency_s} and~\eqref{eq:milp:McCormick_linearization} are satisfied.
    Furthermore, we have $\tilde{\bfdelta} \in \Deltaml^T$.
    Finally, we prove that equalities~\eqref{eq:milp:Valid_cuts_pomdp} are satisfied. By definition of a probability distribution, we directly see that constraints~\eqref{eq:milp:Valid_cuts_pomdp_consistency1} are satisfied.
    We prove \eqref{eq:milp:Valid_cuts_pomdp_consistency2} and \eqref{eq:milp:Valid_cuts_pomdp_main}. We compute the left-hand side of \eqref{eq:milp:Valid_cuts_pomdp_consistency2}:
    \begin{align*}
        &\sum_{a \in \calA} \mu_{s'a'soa}^t = \sum_{a \in \calA} \bbP_{\bfdelta}(S_{t-1}=s',A_{t-1}=a',S_t=s,O_t=o,A_t=a)\\
        &= \sum_{a \in \calA} \sum_{\substack{s_0,\ldots,s_{t-2} \\ h_{t-1}}} \\
        &\bbP_{\bfdelta}((S_i=s_i,O_i=o_i,A_i=a_i)_{0\leq i\leq t-2},S_{t-1}=s',O_{t-1}=o',A_{t-1}=a',S_t=s,O_t=o,A_t=a)\\
        &= p(o|s)p(s|s',a')\sum_{\substack{s_0,\ldots,s_{t-2} \\ h_{t-1}}} \bbP_{\bfdelta}((S_i=s_i,O_i=o_i,A_i=a_i)_{0\leq i\leq t-2},S_{t-1}=s',O_{t-1}=o_{t-1},A_{t-1}=a') \\
        & \sum_{a \in \calA} \delta_{a|h_t} \\
        &= p(o|s)p(s|s',a')\sum_{\substack{s_0,\ldots,s_{t-2} \\ h_{t-1}}} \bbP_{\bfdelta}((S_i=s_i,O_i=o_i,A_i=a_i)_{0\leq i\leq t-2},S_{t-1}=s',O_{t-1}=o',A_{t-1}=a')\\
        &=  p(o|s)p(s|s',a') \bbP_{\bfdelta}(S_{t-1}=s',A_{t-1}=a')\\
        &=  p(o|s)\mu_{s'a's}^{t-1}
    \end{align*}
    where we used the definition of the probability distribution~\eqref{eq:milp:proof_distrib} at the third equation. Therefore, constraints~\eqref{eq:milp:Valid_cuts_pomdp_consistency2} are satisfied by $\bfmu$. 
    To prove that constraints~\eqref{eq:milp:Valid_cuts_pomdp_main} are satisfied, we prove that
    $$\bbP_{\bfdelta}(S_t=s_t| S_{t-1}=s_{t-1},A_{t-1}=a_{t-1},O_t=o_t,A_t=a_t) = \bbP_{\bfdelta}(S_t=s_t|  S_{t-1}=s_{t-1},A_{t-1}=a_{t-1},O_t=o_t)$$
    We compute $\bbP_{\bfdelta}(S_t=s_t| S_{t-1}=s',A_{t-1}=a',O_t=o,A_t=a)$:
    \begin{align*}
        &\bbP_{\bfdelta}(S_t=s_t| S_{t-1}=s_{t-1},A_{t-1}=a_{t-1},O_t=o_t,A_t=a_t) \\
        &= \frac{\bbP_{\bfdelta}(S_{t-1}=s_{t-1},A_{t-1}=a_{t-1},S_t=s_t,O_t=o_t,A_t=a_t)}{\bbP_{\bfdelta}(S_{t-1}=s_{t-1},A_{t-1}=a_{t-1},O_t=o_t,A_t=a_t)}\\
        &= \frac{\sum_{\substack{s_0,\ldots,s_{t-2} \\ h_{t-1}}} \bbP_{\bfdelta}((S_i=s_i,O_i=o_i,A_i=a_i)_{0\leq i\leq t})}{\sum_{\substack{s_0,\ldots,s_{t-2},s_t \\ h_{t-1}}} \bbP_{\bfdelta}((S_i=s_i,O_i=o_i,A_i=a_i)_{0\leq i\leq t})} \\
        &= \frac{\sum_{\substack{s_0,\ldots,s_{t-2} \\ h_{t-1}}} \delta_{a_t|h_t}^t p(o_t|s_t)p(s_{t}|s_{t-1},a_{t-1}) \bbP_{\bfdelta}((S_i=s_i,O_i=o_i,A_i=a_i)_{0\leq i\leq t-1})}{\sum_{\substack{s_0,\ldots,s_{t-2},s_t' \\ h_{t-1}}} \delta_{a_t|h_t}^t p(o_t|s_t)p(s_{t}|s_{t-1},a_{t-1}) \bbP_{\bfdelta}((S_i=s_i,O_i=o_i,A_i=a_i)_{0\leq i\leq t-1})}
    \end{align*}
    \begin{align*}
        &= \frac{p(o_t|s_t)p(s_{t}|s_{t-1},a_{t-1}) \sum_{\substack{s_0,\ldots,s_{t-2} \\ h_{t-1}}} \delta_{a_t|h_t}^t \bbP_{\bfdelta}((S_i=s_i,O_i=o_i,A_i=a_i)_{0\leq i\leq t-1})}{\sum_{s_t'} p(o_t|s_t')p(s_{t}'|s_{t-1},a_{t-1})\sum_{\substack{s_0,\ldots,s_{t-2} \\ h_{t-1}}} \delta_{a_t|h_t}^t \bbP_{\bfdelta}((S_i=s_i,O_i=o_i,A_i=a_i)_{0\leq i\leq t-1})} \\
        &= \frac{p(o_t|s_t)p(s_{t}|s_{t-1},a_{t-1})}{\sum_{s_t'} p(o_t|s_t')p(s_{t}'|s_{t-1},a_{t-1})}
    \end{align*}
    where the last line goes from the fact that the term $\delta_{a_t|h_t}^t \bbP_{\bfdelta}((S_i=s_i,O_i=o_i,A_i=a_i)_{0\leq i\leq t-1})$ does not depend on $s_t$.
    Hence, constraints~\eqref{eq:milp:Valid_cuts_pomdp_main} are satisfied by $\bfmu$. We deduce that $\bfmu$ is a feasible solution of MILP~\eqref{pb:milp:MILP_pomdp} satisfying the valid inequalities~\eqref{eq:milp:Valid_cuts_pomdp}.
    Therefore,
    \begin{align*}
        \bbE_{\bfdelta} \left[ \sum_{t=0}^T r_t(S_t,A_t) | S_0 \sim \bfb \right] &= \sum_{t=0}^T\sum_{s,a}\bbP_{\bfdelta}(S_t=s,A_t=a)r_t(s,a) \leq z_{\rmRc}^{T}(\bfb)
    \end{align*} 
    By maximizing over $\bfdelta$ the left-hand side, we obtain $v^{T} \leq z_{\rm{R}^{\rm{c}}}^{T}$. It achieves the proof.
\end{proof}

\begin{proof}[Proof of Proposition~\ref{prop:computingSMFApproximationWithFiniteHorizon}]
    We need to prove that Problem~\eqref{eq:memorylessFromBeliefValueFunction} is equivalent to Problem~\eqref{pb:POMDPmlfiniteHorizon}. 
    The proof relies on the fact that solving Problem~\eqref{eq:memorylessFromBeliefValueFunction} is equivalent to solving a POMDP with memoryless policies over $T$ time steps
    \begin{align*}
         \hat{v}_{\rmml}^T(\bfb)&:=  \max_{a' \in \calA}\max_{\bfdelta \in \Deltaml^{T}} \bbE_{\bfdelta}\left[ \sum_{t=0}^{T} \gamma^{t}r(S_{t},A_{t}) + \gamma^{T+1}v_{\rmMDP}(S_{T+1}) \Big| S_0 \sim \bfb, A_0=a'\right]\\
         &= \max_{a' \in \calA} \max_{\substack{\bfdelta \in \Deltaml^{T}\colon \\ \delta_{a'|o}^0=1 \forall o \in \calO}} \bbE_{\bfdelta}\left[ \sum_{t=0}^{T} \gamma^{t}r(S_{t},A_{t}) + \gamma^{T+1}v_{\rmMDP}(S_{T+1}) \Big| S_0 \sim \bfb, A_0=a' \right] \\
         &= \max_{a' \in \calA} \max_{\substack{\bfdelta \in \Deltaml^{T}\colon \\ \delta_{a'|o}^0=1 \forall o \in \calO}} \bbE_{\bfdelta}\left[ \sum_{t=0}^{T} \gamma^{t}r(S_{t},A_{t}) + \gamma^{T+1}v_{\rmMDP}(S_{T+1}) \Big| S_0 \sim \bfb, A_0=a' \right]\\
         &= \max_{\substack{\bfdelta \in \Deltaml^{T}\colon \\ \delta_{a|o}^0=\delta_{a|o'}^0 \\ \forall o,o' \in \calO, \forall a \in \calA}} \bbE_{\bfdelta}\left[ \sum_{t=0}^{T} \gamma^{t}r(S_{t},A_{t}) + \gamma^{T+1}v_{\rmMDP}(S_{T+1}) \Big| \bfb\right]
    \end{align*}
    We introduce the reward function $\tilde{r}_t \colon \calS \times \calA \rightarrow \bbR$ such that $\tilde{r}_t(s,a) = \gamma^t r(s,a)$ if $t < T$ and $\tilde{r}_{T}(s,a) = \gamma^T r(s,a) + \gamma^{T+1}\sum_{s'} p(s'|s,a) v_{\rmMDP}(s')$, for all $s$ in $\calS$ and $a$ in $\calA$. Then, it follows that the last line corresponds to a memoryless POMDP over a finite horizon $T$, the reward function $\tilde{r}$ and the initial probability distribution is the belief state $\bfb$. 
    Theorem~\ref{theo:milp:NLP_optimal_solution} combined with McCormick's linearization~\eqref{eq:milp:McCormick_linearization} ensures that the optimization problem can be written:
    \begin{subequations}
        \begin{alignat*}{2}
            \hat{v}_{\rmml}^T(\bfb) = \max_{\bfmu,\bfdelta}  \enskip & \sum_{t=0}^T \sum_{\substack{s \in \calS, a \in \calA}} \tilder_t(s,a)\mu_{sa}^t   & \quad &\\
            \mathrm{s.t.} \enskip 
            & \delta_{a|o}^0 = \delta_{a|o'}^0 & \forall a \in \calA, o,o' \in \calO \\
            & \mu_s^0 = b(s) & \forall s \in \calS  \\
            & (\bfmu,\bfdelta) \ \mathrm{satisfies}~\eqref{eq:milp:NLP_state_initial}-\eqref{eq:milp:NLP_consistency_s},~\eqref{eq:milp:McCormick_linearization}
        \end{alignat*}
    \end{subequations}
    The dependence of the reward in time does not impact the result of Theorem~\ref{theo:milp:NLP_optimal_solution}. It achieves the proof.
\end{proof}

\MDPequivalenceInfinitehorizon*

\begin{proof}[Proof of Theorem~\ref{theo:milp:MDP_approx_equivalence_infinite}]
    Let $\bfb$ be a belief in $\calB$ and $T,T'$ in $\bbZ_{+}$.
    We recall that the MILPs and LPs are solved on the POMDP instance specified by the reward function $\tilder$.
    We start by proving Inequality~\eqref{eq:milp:inequality_information}.
    First, since $\bfdelta_{\rmSMF}$ is a history-dependent policy, the first inequality $v_{\rmSMF(T)}(\bfb) \leq v^*(\bfb)$ holds.

    Second, the inequality $\tildez_{\rmRc}^{T}(\bfb) \leq \tildez_{\rmR}^{T}(\bfb)$ holds because the left-hand side is obtained by using more constraints in the MILP solved.

    Finally it suffices to prove the inequality $v^*(\bfb) \leq \tildez_{\rmRc}^{T'}(\bfb)$. To do so, we prove that $v^*(\bfb) \leq v^{T'}(\bfb)$. Then, Theorem~\ref{theo:milp:MDP_approx_equivalence_infinite} will ensure that $v^{T'}(\bfb) \leq \tildez_{\rmRc}^{T'}(\bfb)$, which will achieve the proof. 
    Given an history-dependent policy $\bfdelta$, we denote by $\Pi_{\bfdelta}^{t}(\bfb_0)$ the marginal probability distribution of state $S_{t}$ according to policy $\bfdelta$. For all $s$ in $\calS$,
    \begin{align*}
        \Pi_{\bfdelta}^{t}(\bfb)_s = \sum_{\substack{s_0,\ldots,s_{t-1} \in \calS \\ o_1,\ldots,o_{t-1} \in \calO \\  a_0,\ldots,a_{t-1} \in \calA}} \bfb(s_0) \delta_{a_0}^0 p(s_{1}|s_{0},a_{0}) \prod_{t'=1}^{t-1} p(o_{t'}|s_{t'})\delta_{a_{t'}|h_{t'}} p(s_{t'+1}|s_{t'},a_{t'}),
    \end{align*}
    where the initial policy $\delta^0$ does not depend on the initial observation because this information is contained in the initial belief $\bfb$.

    Now we rewrite $v^*(\bfb)$ using $\Pi_{\bfdelta}$:
    \begin{align*}
        v^*(\bfb) &= \max_{\bfdelta \in \Deltahis} \bbE_{\bfdelta}\left[\sum_{t=0}^{\infty} \gamma^t r(S_t,A_t) | \bfb \right] \\
        &= \max_{\bfdelta \in \Deltahis} \bbE_{\bfdelta}\left[\sum_{t=0}^{T'} \gamma^t r(S_t,A_t) | \bfb \right] + \bbE_{\bfdelta}\left[\sum_{t=T+1}^{\infty} \gamma^t r(S_t,A_t) | \bfb \right]  \\
        &= \max_{\bfdelta \in \Deltahis} \bbE_{\bfdelta}\left[\sum_{t=0}^{T'} \gamma^t r(S_t,A_t) | \bfb \right] + \gamma^{T'+1} \sum_{s\in \calS} \Pi_{\bfdelta}^{T'+1}(\bfb)_s \underbrace{\bbE_{\bfdelta}\left[\sum_{t=0}^{\infty} \gamma^t r(S_t,A_t) | S_0=s \right]}_{\leq v_{\rmMDP}(s)} \\
        &\leq \max_{\bfdelta \in \Deltahis^{T'}} \bbE_{\bfdelta}\left[\sum_{t=0}^{T'} \gamma^t r(S_t,A_t) + \gamma^{T'+1} v_{\rmMDP}(S_{T'+1})  | \bfb \right]= v^{T'}(\bfb)
    \end{align*}

    Now we prove Inequality~\eqref{eq:milp:decreasing_bounds}. To do so we prove that any feasible solution of the linear relaxation of MILP~\eqref{pb:milp:MILP_pomdp} with finite horizon $T+1$ is a feasible solution of the linear relaxation of MILP~\eqref{pb:milp:MILP_pomdp} with finite horizon $T$.
    We denote respectively by $\rmPc^{T}$ and $\rmPc^{T+1}$ the feasible set of linear relaxation of MILP~\eqref{pb:milp:MILP_pomdp} with finite horizon $T$ and $T+1$.
    Let $(\bfmu,\bfdelta)$ be an optimal solution of $\rmPc^{T+1}$. We will prove that we can construct a feasible solution $(\tilde{\bfmu},\tilde{\bfdelta})$ in $\rmPc^T$ with an objective function $\tildez^T$ higher than the optimal value $\tildez_{\rmRc}^{T+1}(\bfb)$.
    Note again that the policy $\bfdelta$ does not play a role in the linear relaxation of MILP~\eqref{pb:milp:MILP_pomdp}.

    We define $(\tilde{\bfmu},\tilde{\bfdelta})$ as follows:
    \begin{align*}
        (\mutilde_{s}^t,\mutilde_{sa}^t,\mutilde_{soa}^t, \deltatilde_{a|o}^t)_{1\leq t \leq T} = (\mu_{s}^t,\mu_{sa}^t,\mu_{soa}^t,\delta_{a|o})_{1\leq t \leq T}
    \end{align*}
    By definition, $(\tilde{\bfmu},\tilde{\bfdelta})$ is a feasible solution of $\rmPc^T$.
    We write the objective value of such solution:
    \begin{align*}
        \tildez^T &= \sum_{t=0}^T \sum_{s,a} \gamma^t r(s,a) \mutilde_{sa}^t  + \gamma^{T+1} \sum_{s} v_{\rmMDP}(s) \mutilde_s^{T+1} \\
                                            &= \sum_{t=0}^T \sum_{s,a} \gamma^t r(s,a) \mutilde_{sa}^t  + \gamma^{T+1} \sum_{s} \mutilde_s^{T+1} \max_{a \in \calA} \left( r(s,a) + \gamma \sum_{s'} p(s'|s,a) v_{\rmMDP}(s')\right) \\
                                            &= \sum_{t=0}^T \sum_{s,a} \gamma^t r(s,a) \mutilde_{sa}^t  + \gamma^{T+1} \max_{\substack{(\mu_{sa})\colon \\ \sum_{a}\mu_{sa} = \mutilde_{s}^{T+1}}} \sum_{s,a} \mu_{sa}  \left( r(s,a) + \gamma \sum_{s'} p(s'|s,a) v_{\rmMDP}(s')\right) \\
    \end{align*}
    The second equality comes from the optimal Bellman equation for MDP, and, the last equality come from a rewriting of the operand $\sum_{s} \max_{a}$ as $\max_{\substack{(\delta_{a|s})_{a,s}\\ \sum_{a} \delta_{a|s}=1}} \sum_{s,a} \delta_{a|s}$, which is equivalent as using variable $\mu_{sa} = \delta_{a|s}\mu_{s}^{T+1}$.
    Since the solution $\mu_{sa}^{T+1}$ satisfies the constraints $\sum_{a}\mu_{sa}^{T+1} = \mutilde_{s}^{T+1}$, we obtain that:
    \begin{align*}
        \tildez^T &\geq \sum_{t=0}^T \sum_{s,a} \gamma^t r(s,a) \mutilde_{sa}^t  + \gamma^{T+1} \sum_{s,a} \mu_{sa}^{T+1} r(s,a) + \gamma^{T+2} \sum_{s'} \overbrace{\sum_{s,a} \mu_{sa}^{T+1} p(s'|s,a)}^{\mu_{s'}^{T+2}} v_{\rmMDP}(s') \\
                                            &= \sum_{t=0}^{T+1} \sum_{s,a} \gamma^t r(s,a) \mutilde_{sa}^t + \gamma^{T+2} \sum_{s'} v_{\rmMDP}(s') \mu_{s'}^{T+2} = \tildez_{\rmRc}^{T+1}(\bfb)
    \end{align*}
    It achieves to prove Inequality~\eqref{eq:milp:decreasing_bounds}.

    Finally, we prove that for any $T$ in $\bbZ_{+}$, $\tildez_{\rmR}^{T}(\bfb) = \tildez_{\rmR}^{0}(\bfb)$. To do so, we show that $\tildez_{\rmR}^{T}(\bfb) = \tildez_{\rmR}^{T+1}(\bfb)$ for any $T$ in $\bbZ_{+}$.
    Let $T$ in $\bbZ_{+}$.
    Theorem~\ref{theo:milp:MDP_approx_equivalence_finiteHorizon} ensures that the linear relaxation of MILP~\eqref{pb:milp:MILP_pomdp} is equivalent to the following linear formulation:
    \begin{subequations}
        \begin{alignat*}{2}
            \tildez_{\rmR}^T(\bfb) = \max_{\bfmu}  \enskip & \sum_{t=0}^T \sum_{s,a} \gamma^t r(s,a) \mu_{sa}^t + \gamma^{T+1} \sum_{s} v_{\rmMDP}(s) \mu_{s}^{T+1}  & \quad & \\
            \mathrm{s.t.} \enskip
             & \bfmu \text{ feasible  solution of Problem}~\eqref{pb:app_proof:LP_mdp}.
        \end{alignat*}
    \end{subequations}
    In the remaining of the proof, we say that $\bfmu \in \rmP^T$ when $\bfmu$ feasible solution of Problem~\eqref{pb:app_proof:LP_mdp} with finite horizon $T$.
    Then, it follows that:
    \begin{align*}
        \tildez_{\rmR}^T(\bfb) &= \max_{\bfmu \in \rmP^T} \sum_{t=0}^T \sum_{s,a} \gamma^t r(s,a) \mu_{sa}^t + \gamma^{T+1} \sum_{s} \mu_{s}^{T+1}  \max_{a} \left( r(s,a) + \gamma \sum_{s'} p(s'|s,a) v_{\rmMDP}(s') \right) \\
                         &= \max_{\bfmu \in \rmP^T} \sum_{t=0}^T \sum_{s,a} \gamma^t r(s,a) \mu_{sa}^t + \gamma^{T+1} \max_{\substack{(\mu_{sa})\colon \\ \sum_{a} \mu_{sa}=\mu_s^{T+1}}} \sum_{s,a} \mu_{sa}\left(r(s,a)  + \gamma \sum_{s'} p(s'|s,a) v_{\rmMDP}(s') \right) \\
                         &= \max_{\bfmu \in \rmP^T} \sum_{t=0}^T \sum_{s,a} \gamma^t r(s,a) \mu_{sa}^t + \gamma^{T+1} \max_{\substack{(\mu_{sa})\colon \\ \sum_{a} \mu_{sa}=\mu_s^{T+1}}} \sum_{s,a} \mu_{sa}\left(r(s,a)  + \gamma \sum_{s'} p(s'|s,a) v_{\rmMDP}(s') \right) \\
                         &= \max_{\bfmu \in \rmP^{T+1}} \sum_{t=0}^{T+1} \sum_{s,a} \gamma^t r(s,a) \mu_{sa}^t + \gamma^{T+2} \sum_{s}  v_{\rmMDP}(s) \mu_{s}^{T+2} = \tildez_{\rmR}^{T+1}(\bfb)
    \end{align*}

\end{proof}

\bibliographystyle{plainnat}
\bibliography{pomdp}

\end{document}